\documentclass[11pt]{amsart}

\usepackage{amsmath,amssymb,amsthm}
\usepackage{array}
\usepackage{booktabs}
\usepackage{enumitem}
\usepackage[margin=1.08in]{geometry}
\usepackage[pdftex,colorlinks,citecolor=black,linkcolor=black,urlcolor=black,
bookmarks=true]{hyperref}
\usepackage{mathtools}
\usepackage{microtype}
\newtheorem{theorem}{Theorem}[section]
\newtheorem{proposition}[theorem]{Proposition}
\newtheorem{lemma}[theorem]{Lemma}
\newtheorem{corollary}[theorem]{Corollary}
\numberwithin{equation}{section}
\theoremstyle{definition}
\newtheorem{definition}[theorem]{Definition}
\theoremstyle{remark}
\newtheorem{remark}[theorem]{Remark}

\DeclareMathOperator{\tr}{tr}
\DeclareMathOperator{\Tr}{Tr}
\DeclareMathOperator{\Ent}{Ent}
\DeclareMathOperator{\End}{End}
\DeclareMathOperator{\Hom}{Hom}
\DeclareMathOperator{\Ind}{Ind}
\DeclareMathOperator{\Res}{Res}
\DeclareMathOperator{\Irr}{Irr}
\DeclareMathOperator{\rank}{rank}
\DeclareMathOperator{\ran}{ran}

\newcommand{\CC}{\mathbb C}
\newcommand{\RR}{\mathbb R}
\newcommand{\cE}{\mathcal E}
\newcommand{\cK}{\mathcal K}
\newcommand{\iu}{\mathrm i}

\title[The half-rate LP bound for binary codes]
{The half-rate linear programming bound for binary codes is
$\frac12-\frac1\pi$}
\author{Andrew Salmon}
\address{Independent Researcher}
\email{asalmon@alum.mit.edu}
\date{September 30, 2026}

\begin{document}

\begin{abstract}
In their work on sphere packing and the modular bootstrap,
Afkhami-Jeddi, Cohn, Hartman, de Laat, and Tajdini conjectured the exact
high-dimensional exponent of the Cohn--Elkies sphere-packing linear program.
OpenAI's Chapter~1 subsequently proved their conjecture by establishing that
both Fourier sign-uncertainty radii are
$(1/\pi+o(1))\sqrt d$.  We prove the binary coding analogue:
\[
 R_D\!\left(\frac12-\frac1\pi\right)=\frac12.
\]
We also formulate the two Krawtchouk sign-uncertainty problems and determine
both of their asymptotics.  If $A^{\mathrm K}_{\pm}(n)$ denotes the smallest
radius $r$ for which a nonzero Krawtchouk $(\pm1)$-eigenfunction $f$ exists with $f(0) = 0$ and
$f(x) \ge 0$ for all $|x| \ge r$, then
\[
 \frac{A^{\mathrm K}_{\pm}(n)}n\longrightarrow
 \frac12-\frac1\pi.
\]
The lower bound proves a mass-concentration principle similar to OpenAI's
Chapter~1 for Hamming space.

The upper bound, on the other hand, follows the approach of the spherical-code construction in OpenAI's Chapter~2.  Gay,
Jeronimo, and Liu formulated a hierarchy for binary codes analogous to
the spherical-code construction and improved the best known binary coding rate bounds by evaluating
the first level of the corresponding hierarchy.
We prove that this hierarchy bounds the Delsarte program
and give a construction at arbitrarily deep levels of the hierarchy,
attaining the upper bound in the limit.  The construction is an $N$-qubit generalization of the
pure-state channel of Alrabiah and Guruswami.
\end{abstract}

\maketitle
\begingroup
\small
\tableofcontents
\endgroup

\section{Introduction}

Let $K_j^{(n)}(i)$ denote the binary Krawtchouk polynomial, normalized by
\begin{equation}
 \sum_{j=0}^n K_j^{(n)}(i)z^j=(1+z)^{n-i}(1-z)^i.
 \label{eq:Krawtchouk-generating-function}
\end{equation}
For an integer distance $d$, the radial Delsarte linear program
\cite{Delsarte1973} is
\begin{equation}
 \operatorname{LP}_n(d)=
 \max\left\{\sum_{i=0}^n A_i:
 \begin{array}{l}
 A_0=1,\quad A_i\geq0,\quad A_i=0\quad(1\leq i<d),\\
 \displaystyle\sum_{i=0}^nA_iK_j^{(n)}(i)\geq0
       \quad(0\leq j\leq n)
 \end{array}\right\}.
 \label{eq:Delsarte-primal}
\end{equation}
Passing to the asymptotic rate, for $0<\delta<1/2$, define the quantity
\begin{equation}
 R_D(\delta)=\limsup_{n\to\infty}\frac1n
 \log_2\operatorname{LP}_n(\lceil\delta n\rceil),
 \label{eq:Delsarte-exponent}
\end{equation}
and for $0<R<1$, define the associated inverse distance bound by
\begin{equation}
 \delta_D(R)=\sup\left\{\delta\in\left[0,\frac12\right]:
 R_D(\delta)\geq R\right\}.
 \label{eq:inverse-Delsarte-distance}
\end{equation}
It is well-known that the Delsarte linear program is an upper bound on the size of a codebook for binary codes in Hamming space for this minimum distance.  That is, if we define the maximum size of the codebook
\begin{equation}
 A_2(n,d)=\max\left\{|C|:C\subseteq\{0,1\}^n,
 d_H(x,y)\geq d\ \text{for all distinct }x,y\in C\right\}
 \label{eq:binary-code-size}
\end{equation}
and define the asymptotic rate of binary codes by
\begin{equation}
 R_2(\delta)=\limsup_{n\to\infty}\frac1n
 \log_2 A_2(n,\lceil\delta n\rceil),
 \label{eq:binary-code-rate}
\end{equation}
the distance distribution of every code is feasible for the Delsarte program,
so $A_2(n,d)\leq\operatorname{LP}_n(d)$ and $R_2(\delta)\leq R_D(\delta)$.  Despite its status as a limit of linear programs, it is still challenging to give computations of the asymptotic rate $R_D(\delta)$.

Our main theorem gives the first exact determination of $R_D(\delta)$ at some point $\delta \in (0,1/2)$.  Our primary interest is the half-rate linear programming bound, denoted $\delta_D(1/2)$.  Our main theorem is as follows.
\begin{theorem}[half-rate linear programming bound]
\label{thm:main}
With
\begin{equation}
 \delta_\pi=\frac12-\frac1\pi,
 \label{eq:delta-pi}
\end{equation}
we have
\begin{equation}
 \boxed{R_D(\delta_\pi)=\frac12,}
 \label{eq:main}
\end{equation}
while $R_D(\delta)<\frac12$ for every $\delta>\delta_\pi$.
Equivalently,
\begin{equation}
 \delta_D\!\left(\frac12\right)=\delta_\pi
 =\frac12-\frac1\pi.
 \label{eq:main-inverse}
\end{equation}
\end{theorem}

The half rate for binary codes is distinguished in a few ways.  For example, self-dual codes have rate $1/2$.  In the sphere packing context, a $(-1)$-eigenfunction uncertainty principle is closely related to the linear programming bound~\cite{CohnGoncalves2019}.  We define a Krawtchouk polynomial analogue of this uncertainty principle here.  For $\varsigma\in\{-1,+1\}$, let $A^{\mathrm K}_{\varsigma}(n)$ be the
smallest radial layer $r$ for which there is a nonzero $\varsigma$-eigenvector
of the unitary radial Krawtchouk transform that vanishes at layer zero and is
nonnegative on every layer $r,\ldots,n$.  Section~\ref{sec:half-disk} gives
the precise definition and proves the lower bound for both signs.

\begin{theorem}[Krawtchouk sign uncertainty]
\label{thm:Krawtchouk-sign-uncertainty}
For each $\varsigma\in\{-1,+1\}$,
\begin{equation}
 \lim_{n\to\infty}\frac{A^{\mathrm K}_{\varsigma}(n)}n
 =\frac12-\frac1\pi.
 \label{eq:Krawtchouk-sign-radius}
\end{equation}
\end{theorem}

The inspiration for the above theorem comes from the Euclidean analogue,
which was recently resolved.
Let $\operatorname{LP}^{\rm CE}_d$ be the Cohn--Elkies bound on the density
of sphere packing in $\RR^d$, and let $A_{\varsigma}(d)$ be the last-sign change radius
for Fourier $\varsigma$-eigenfunctions, with $\varsigma\in\{-1,+1\}$
\cite{GoncalvesOliveiraeSilvaSteinerberger2017,CohnGoncalves2019}.
Afkhami-Jeddi, Cohn, Hartman, de Laat, and Tajdini conjectured the following limits
in their study of sphere packing and the modular bootstrap
\cite[Conjecture~3.2]{AfkhamiJeddiEtAl2020}:
\begin{equation}
 \lim_{d\to\infty}
 \bigl(\operatorname{LP}^{\rm CE}_d\bigr)^{1/d}
 =\sqrt{\frac{e}{2\pi}},
 \qquad
 \lim_{d\to\infty}\frac{A_-(d)}{\sqrt d}=\frac1\pi.
 \label{eq:Euclidean-Chapter-One-theorem}
\end{equation}
OpenAI's Chapter~1 proved these limits and also proved the same asymptotic rate
for $A_+(d)$ in Theorem~1.2~\cite[Chapter~1]{OpenAI2026}.

We give both a high-dimensional $(+1)$ and $(-1)$-eigenfunction Krawtchouk uncertainty theorem.
A conic duality type statement relates $A^{\mathrm K}_-(n)$ with the Delsarte bound at
rate $1/2$, and similar to the Euclidean case~\cite[Chapter~1, Appendix~A]{OpenAI2026}, we have $A^{\mathrm K}_+(n) \le A^{\mathrm K}_-(n)$ for all $n$.  Thus, despite the fact that Hamming space has no analogue of the Euclidean rescaling that directly
connects a Cohn--Elkies auxiliary function to a Fourier $(-1)$-eigenfunction, we arrive at analogues of the Euclidean results
because feasible solutions to these problems at the rate $1/2$ behave like the weight enumerators of self-dual codes.

The previous lower bounds were the Gilbert--Varshamov bound, later improved
by Samorodnitsky's bound.  We write
\begin{equation}
 \begin{aligned}
 h_2(x)&=-x\log_2x-(1-x)\log_2(1-x),
 &R_{\rm GV}(\delta)&=1-h_2(\delta),\\
 M_1(\delta)&=h_2\!\left(\frac12-\sqrt{\delta(1-\delta)}\right).
 \end{aligned}
 \label{eq:prior-lower-notation}
\end{equation}
The Gilbert--Varshamov construction gives
$R_D(\delta)\geq R_2(\delta)\geq R_{\rm GV}(\delta)$
\cite{Gilbert1952,Varshamov1957}.  Samorodnitsky's lower bound on the
Delsarte optimum strengthens this bound to the arithmetic
mean of the Gilbert--Varshamov and first
McEliece--Rodemich--Rumsey--Welch (MRRW) exponents,
\begin{equation}
 L_{\rm Sam}(\delta)
 =\frac12\bigl(R_{\rm GV}(\delta)+M_1(\delta)\bigr),
 \label{eq:prior-lower-bounds}
\end{equation}
and it gives $R_D(\delta)\geq L_{\rm Sam}(\delta)$
\cite{Samorodnitsky2001}.  At $1/2-1/\pi$,
\begin{equation}
 R_{\rm GV}(\delta_\pi)=0.3162395419\ldots,
 \qquad L_{\rm Sam}(\delta_\pi)=0.4146594612\ldots.
 \label{eq:prior-lower-values}
\end{equation}

For the upper bound, the most recent previous bound is due to Gay, Jeronimo, and Liu.  The
history of asymptotic bounds at $\delta_\pi$ is shown below, with decimals
indicating numerical computations.
\begin{center}
\small
\renewcommand{\arraystretch}{1.12}
\begin{tabular}{@{}ll@{}}
\toprule
dual construction & evaluated upper exponent at $\delta_\pi$ \\
\midrule
first MRRW & $0.5130793807\ldots$ \\
second MRRW & $0.5017797958\ldots$ \\
OpenAI Chapter~2 combined moving-subspace bound & $0.5005547387\ldots$ \\
Gay--Jeronimo--Liu first Honeycomb level & $0.5003477743\ldots$ \\
this paper & $1/2$ exactly \\
\bottomrule
\end{tabular}
\end{center}
The MRRW entries are the classical first and second bounds \cite{MRRW1977}.
OpenAI Chapter~2 then improved the second MRRW bound by
combining moving subspace constructions for Hamming space and the Johnson scheme
\cite[Chapter~2]{OpenAI2026}.  At the point $\delta_\pi$, the bound of Gay, Jeronimo, and Liu's first
Honeycomb exponent improves numerically on the OpenAI exponent
\cite[Theorem~1.2 and Appendix~C.1]{GayJeronimoLiu2026}.
Gay--Jeronimo--Liu also identify this exponent with
a $2\times2$ Horn-channel exponent
\cite[Section~8]{GayJeronimoLiu2026}.

To prove the lower bound, we will use a mass concentration inequality
to give a lower bound for the $(-1)$-eigenfunction
uncertainty principle.  The upper bound uses the strategy of Chapter~2, where
in the Euclidean case, all levels of the hierarchy of linear programming bounds
given by the moving subspace bound must be used at once.

There are new difficulties that emerge in the binary code case that are not present
in the Euclidean case, and vice versa.  The lower bound has fewer analytic
difficulties in the binary coding setting.  Chapter~1 obtains its mass concentration statement through
Mellin-strip interpolation, gamma-factor asymptotics, and shifted Mellin
inversion.  Here the Krawtchouk generating polynomial is controlled by a
single Dirichlet problem on a half-disk in the complex plane.
Solving this harmonic boundary-value problem gives precise control
over the values of the polynomial, and to get control over the coefficients,
we may use Cauchy's formula.

On the other hand, the upper bound involves nontrivial multiplicity spaces
which are not present in the spherical coding case.  To handle these, we generalize
the finite moving-subspace bound to higher multiplicity spaces, where the data
needed to construct a bound is guaranteed by the Perron-Frobenius theorem.
Additionally, we use results of Keyl and Werner to pass to a semiclassical
limit.  This limit was called the Horn-channel hierarchy in~\cite{GayJeronimoLiu2026},
who arrived at it through consideration of the classical-quantum channels
from~\cite{AlrabiahGuruswami2026}.

To help orient the reader, we give the following ``proof sketch.''  For positive
semidefinite matrices $K,L$ with $\tr(K+L)=1$, define
\begin{align}
 \Gamma(K,L)&=2\tr(K^{1/2}L^{1/2}),
 \label{eq:Gamma}\\
 \Phi(K,L)&=\Ent(K)+\Ent(L)-\Ent(K+L),
 \label{eq:Phi}\\
 J(K,L)&=\tr(K^{1/2}L^{1/2})
 +\frac12\tr\!\left(G^{-1/2}QG^{-1/2}Q\right),
 \label{eq:J}
\end{align}
where $G=K+L$, $Q=K-L$,
$\Ent(A)=-\tr(A\log_2A)$, and inverses are taken on $\operatorname{supp}G$.
\begin{equation}
\begin{gathered}
 \bigl(\liminf_{n\to\infty}\tfrac{A^{\mathrm K}_{\pm}(n)}n\ge
 \tfrac12-\tfrac1\pi\bigr)
 \Longrightarrow R_D(\delta_\pi)\geq\tfrac12
 \Longrightarrow
 \bigl(\Gamma>\tfrac2\pi\Longrightarrow\Phi>\tfrac12\bigr)
 \Longrightarrow J\leq\tfrac2\pi,\\[2mm]
 \left(
 J_N=\frac1{2N}\cot\frac{\pi}{4N}
 \right)
 \Longrightarrow R_D(\delta_\pi)\leq\tfrac12.
\end{gathered}
\label{eq:proof-circle}
\end{equation}
The $N$-qubit projection construction in Section~\ref{subsec:N-qubit} gives a construction at level $2^{N+1}$ of the hierarchy, giving the value $J_N$ here.

\section*{Use of generative artificial intelligence}

OpenAI's GPT-5.6 Sol was used to understand Chapters~1 and~2 of
\cite{OpenAI2026}, the work of Alrabiah and Guruswami
\cite{AlrabiahGuruswami2026}, and the work of Gay, Jeronimo, and Liu
\cite{GayJeronimoLiu2026}; to identify the Keyl--Werner spectrum-estimation
reference \cite{KeylWerner2001}; to draft this document; to supply proofs
for Theorem~\ref{thm:mass}; and especially to suggest the $N$-qubit
projection construction and computation of $J$.

The connection between the $N$-qubit projection construction and the pretty good measurement
was found only after the initial proof had been written.  The
initial proof arose instead when Sol was tasked with the goal of supplying a proof of the optimization problem involving $\Gamma$ and $\Phi$ using a
connection to integrable systems and the Yang--Baxter equation.
Several conceptual refinements related to this construction, including Proposition~\ref{prop:reflection-pair} and Lemma~\ref{lem:gaussian-sandwich} arose from discussions with Fable 5.1.
The author has checked the resulting arguments and takes responsibility for all
statements and proofs in the article.

\section{A Hecke-algebra generalization of the finite moving-subspace bound}
\label{sec:Hecke-generalization}

One classical object in coding theory in connection with linear programming bounds is the association scheme.  However, the moving-subspace bound does not naturally fit in the language of association schemes since if our space is $\mathsf G/\mathsf H$, the bound uses a nontrivial representation $E$ of $\mathsf H$.  Thus, we will phrase our results in the language of Hecke algebras throughout this section.  In the case that $(\mathsf G, \mathsf H)$ is a Gelfand pair, then the Hecke algebra for the trivial representation is the Bose-Mesner algebra of the corresponding association scheme.  We will see below that the positive definiteness construction used by the moving-subspace bound falls out naturally from a bilinear trace contraction that we will introduce in Definition~\ref{def:bilinear-trace-contraction}.

Throughout this section, $\mathsf G$ is a finite group,
$\mathsf H\leq\mathsf G$, and $E$ is a finite-dimensional unitary
$\mathsf H$-module.  We allow the spaces
\begin{equation}
 M_\pi=\Hom_{\mathsf H}(E,U_\pi|_{\mathsf H})
 \label{eq:standing-Hecke-multiplicity-spaces}
\end{equation}
to have arbitrary dimension.  Previous versions of the finite moving-subspace theorem generalized
below were given in the multiplicity-one setting in OpenAI's Chapter~2 and
Gay--Jeronimo--Liu
\cite[Chapter~2, Theorem~4.2]{OpenAI2026}
\cite[Theorem~2.1]{GayJeronimoLiu2026}.

\subsection{Hecke algebra preliminaries}

Put $X=\mathsf G/\mathsf H$ with basepoint $o=\mathsf H$.  All Hilbert
spaces below are finite-dimensional, and $X$ carries the uniform probability
measure.

\begin{definition}[homogeneous bundle and induced representation]
\label{def:homogeneous-bundle}
The homogeneous bundle associated with $E$ is
\begin{equation}
 \cE_E=\mathsf G\times_{\mathsf H}E\longrightarrow X
 \label{eq:associated-bundle}
\end{equation}
with $(gh,e)\sim(g,he)$.  Its fiber over $x=g\mathsf H$ is denoted $E_x$.
The Hilbert space of sections is the induced representation
\begin{equation}
 \mathcal I_E=L^2(X,\cE_E)=\Ind_{\mathsf H}^{\mathsf G}E.
 \label{eq:induced-module}
\end{equation}
Equivalently, $\mathcal I_E$ consists of functions $f:\mathsf G\to E$
satisfying $f(gh)=h^{-1}f(g)$, with inner product
\begin{equation}
 \langle f_1,f_2\rangle
 =\frac1{|X|}\sum_{g\mathsf H\in X}
   \langle f_1(g),f_2(g)\rangle_E.
 \label{eq:induced-inner-product}
\end{equation}
The group $\mathsf G$ acts by left translation.
\end{definition}

\begin{definition}[generalized Hecke algebra]
\label{def:generalized-Hecke-algebra}
The generalized Hecke algebra of $(\mathsf G,\mathsf H,E)$ is the
commutant
\begin{equation}
 \mathcal H_E=\End_{\mathsf G}(\mathcal I_E)
 =\{A\in\End(\mathcal I_E):Ag=gA\text{ for every }g\in\mathsf G\}.
 \label{eq:generalized-Hecke-algebra}
\end{equation}
\end{definition}

If $U_\pi$ runs over the irreducible $\mathsf G$-modules and
\begin{equation}
 M_\pi=\Hom_{\mathsf H}(E,U_\pi|_{\mathsf H}),
 \label{eq:Hecke-multiplicity-space}
\end{equation}
then Frobenius reciprocity gives the two complementary descriptions
\begin{equation}
 \mathcal I_E\cong\bigoplus_{\pi\in\Irr(\mathsf G)}
 U_\pi\otimes M_\pi^*,
 \qquad
 \mathcal H_E\cong\bigoplus_{\pi\in\Irr(\mathsf G)}\End(M_\pi).
 \label{eq:Hecke-Wedderburn}
\end{equation}
The number $\dim M_\pi$ is the multiplicity of $U_\pi$ in
$\mathcal I_E$.  On this isotypic summand, equivariance forces an operator
to be the identity on $U_\pi$ but permits an arbitrary endomorphism
$A_\pi$ of the multiplicity factor.

We will need to distinguish two relevant traces.  We use $\tr$ for the trace
in each endomorphism space $\End(M_\pi)$ and $\Tr$ for the trace of the
action on the induced representation $\mathcal I_E$.  For
$A=(A_\pi)_\pi\in\mathcal H_E$, the isotypic decomposition relates these
traces, and likewise the ranks, by
\begin{equation}
 \Tr(A)=\sum_\pi(\dim U_\pi)\tr(A_\pi),\qquad
 \rank A=\sum_\pi(\dim U_\pi)\rank A_\pi.
 \label{eq:Hecke-trace-rank}
\end{equation}

\begin{definition}[$\Omega$-isotypic summand]
\label{def:isotypic-Hecke-summand}
For $\Omega\subset\Irr(\mathsf G)$ let $z_\Omega$ be the central projection
onto the corresponding isotypic summands and put
\begin{equation}
 \mathcal I_{E,\Omega}=z_\Omega\mathcal I_E,
 \qquad
 \mathcal H_{E,\Omega}
 =\End_{\mathsf G}(\mathcal I_{E,\Omega})
 \cong z_\Omega\mathcal H_Ez_\Omega
 =\bigoplus_{\pi\in\Omega}\End(M_\pi).
 \label{eq:isotypic-Hecke-summand}
\end{equation}
Because $z_\Omega$ is central, $\mathcal H_{E,\Omega}$ is a central direct
summand of $\mathcal H_E$, with unit $z_\Omega$.
\end{definition}

\begin{definition}[bundle kernel and positive type]
\label{def:bundle-positive-type}
Every $A\in\mathcal H_E$ has a unique invariant bundle kernel
$K_A(x,y):E_y\to E_x$, normalized by
\begin{equation}
 (Af)(x)=\frac1{|X|}\sum_{y\in X}K_A(x,y)f(y).
 \label{eq:bundle-kernel-action}
\end{equation}
It satisfies
\begin{equation}
 K_A(gx,gy)=gK_A(x,y)g^{-1},\qquad
 K_{A^*}(x,y)=K_A(y,x)^*.
 \label{eq:bundle-kernel-covariance}
\end{equation}
A Hermitian bundle kernel $K$ is of \emph{positive type} if
\begin{equation}
 \sum_{a,b=1}^s\langle v_a,K(x_a,x_b)v_b\rangle\geq0
 \label{eq:bundle-positive-type}
\end{equation}
for every finite collection $x_1,\ldots,x_s\in X$ and
$v_a\in E_{x_a}$.
\end{definition}

Positivity of an operator $A$ is equivalent to the positive type condition on $K_A$
and to positivity of every block $A_\pi\in\End(M_\pi)$.  Moreover,
\begin{equation}
 \Tr(A)=\frac1{|X|}\sum_{x\in X}\tr K_A(x,x).
 \label{eq:kernel-operator-trace}
\end{equation}
We write
\begin{equation}
 \mathcal H_E^+=\{A\in\mathcal H_E:A\succeq0\}
 \label{eq:positive-Hecke-cone}
\end{equation}
for the positive semidefinite elements in the Hecke algebra.
In the Wedderburn decomposition \eqref{eq:Hecke-Wedderburn}, $\mathcal H_E^+$
is the product of the positive semidefinite cones in the blocks $\End(M_\pi)$.
For the trivial $\mathsf H$-module $\mathbf1$, write $\mathcal H_1=\mathcal H_{\mathbf1}$;
this is the algebra of invariant scalar kernels on $X$.

\begin{definition}[positive-type coordinate and threshold graph]
\label{def:threshold-graph}
Choose a unitary $\mathsf G$-module $V$ with a unit
$\mathsf H$-fixed vector $\ell_o$.  For $x=g\mathsf H$, put
$\ell_x=g\ell_o$ and assume that the coordinate
\begin{equation}
 t(x,y)=\langle\ell_x,\ell_y\rangle
 \label{eq:positive-type-coordinate}
\end{equation}
is real.  Then $t$ is $\mathsf G$-invariant, of positive type, and normalized
by $t(x,x)=1$.  For $\tau\in\RR$, let $\mathfrak G_t(\tau)$ be the graph on
$X$ in which distinct $x,y$ are adjacent when $t(x,y)>\tau$.  Thus points
allowed to occur together in an independent set satisfy
\begin{equation}
 t(x,y)\leq\tau.
 \label{eq:allowed-threshold-pairs}
\end{equation}
\end{definition}

\begin{definition}
\label{def:multiplication-map}
For $A\in\mathcal H_E$, define the multiplication-by-$t$ map
$\mathsf S_t:\mathcal H_E\to\mathcal H_E$ by
\begin{equation}
 K_{\mathsf S_t(A)}(x,y)=t(x,y)K_A(x,y).
 \label{eq:Hecke-Schur-kernel}
\end{equation}
The right side is again an invariant bundle kernel.  On the
$\Omega$-isotypic summand, the induced map is
\begin{equation}
 \Theta_\Omega(A)=z_\Omega\mathsf S_t(A)z_\Omega.
 \label{eq:isotypic-multiplication-map}
\end{equation}
\end{definition}

\begin{definition}
Let $\mathcal A$ and $\mathcal B$ be finite-dimensional unital
$*$-algebras.  A linear map $\Psi:\mathcal A\to\mathcal B$ is
\emph{positive} if $\Psi(A)\succeq0$ whenever $A\succeq0$.  It is
\emph{completely positive} if
\begin{equation}
 \operatorname{id}_{M_s}\otimes\Psi:
 M_s(\mathcal A)\longrightarrow M_s(\mathcal B)
 \label{eq:complete-positivity}
\end{equation}
is positive for every $s\geq1$.
\end{definition}

The fixed vector $\ell_o$ defines an $\mathsf H$-intertwining isometry
\begin{equation}
 j_{\ell,o}:E\longrightarrow(\Res_{\mathsf H}^{\mathsf G}V)\otimes E,
 \qquad e\longmapsto\ell_o\otimes e,
 \label{eq:basepoint-coordinate-lift}
\end{equation}
where $\mathsf H$ acts diagonally on the target.  Inducing this map and using
the tensor identity
\begin{equation}
 \Ind_{\mathsf H}^{\mathsf G}
 \bigl((\Res_{\mathsf H}^{\mathsf G}V)\otimes E\bigr)
 \cong V\otimes\Ind_{\mathsf H}^{\mathsf G}E
 \label{eq:induced-tensor-identity}
\end{equation}
gives the $\mathsf G$-intertwiner
\begin{equation}
 J_\ell=\Ind_{\mathsf H}^{\mathsf G}(j_{\ell,o}):
 \mathcal I_E\longrightarrow V\otimes\mathcal I_E,
 \qquad (J_\ell f)(x)=\ell_x\otimes f(x),
 \label{eq:coordinate-lift}
\end{equation}
with the diagonal action on the target.  The map $J_\ell$ is an isometry with
respect to the natural Hilbert space structure.

\begin{proposition}[Stinespring form]
\label{prop:multiplication-Stinespring}
For every $A\in\mathcal H_E$,
\begin{equation}
 \mathsf S_t(A)=J_\ell^*(I_V\otimes A)J_\ell.
 \label{eq:multiplication-Stinespring}
\end{equation}
Consequently, $\mathsf S_t$ is unital and completely positive.
\end{proposition}

\begin{proof}
The $(x,y)$ kernel entry of the right side of
Equation~\eqref{eq:multiplication-Stinespring} is
\[
 \langle\ell_x,\ell_y\rangle K_A(x,y)=t(x,y)K_A(x,y),
\]
which is Equation~\eqref{eq:Hecke-Schur-kernel}.  The map
$A\mapsto I_V\otimes A$ is a unital $*$-representation, and $J_\ell$ is an
isometry.  The displayed compression is therefore unital and completely
positive.
\end{proof}

The map $\mathsf S_t$ is completely positive and unital by the above proposition, and additionally is self-adjoint by the kernel identity
$\Tr(\mathsf S_t(A)B)=\Tr(A\mathsf S_t(B))$.  As a result, $\Theta_\Omega$ is completely positive, subunital, and self-adjoint for
the ordinary operator trace.

\subsection{The generalized moving-subspace bound}

We are now ready to give the moving-subspace construction
valid for any choice of $(\mathsf G, \mathsf H, E, \Omega)$.
The candidates for the eigenvalue $\Lambda$ can still come from the
Perron-Frobenius theorem, generalized to positive maps.

The bilinear trace contraction given below allows us to transport
our constructions from $\mathcal{H}_E^+$ to a kernel on
$\mathsf{G} / \mathsf{H}$.

\begin{definition}[bilinear trace contraction]
\label{def:bilinear-trace-contraction}
For two bundle kernels $K_C,K_D$, define
\begin{equation}
 \mathcal T_E(C,D)(x,y)
 =\tr\bigl(K_C(x,y)K_D(y,x)\bigr).
 \label{eq:bilinear-trace-contraction}
\end{equation}
Diagrammatically,
\[
 \mathcal T_E(C,D)(x,y)
 =\tr\Bigl(E_x\xrightarrow{K_D(y,x)}E_y\xrightarrow{K_C(x,y)}E_x\Bigr).
\]
\end{definition}

The contraction preserves these positive cones:
\begin{equation}
 \mathcal T_E:\mathcal H_E^+\times\mathcal H_E^+
 \longrightarrow\mathcal H_1^+.
 \label{eq:trace-contraction-positive-cones}
\end{equation}
Indeed, on any finite collection of points, write
$K_C(x_i,x_j)=R_i^*R_j$ and $K_D(x_i,x_j)=S_i^*S_j$.  Then
\begin{equation}
 \mathcal T_E(C,D)(x_i,x_j)
 =\langle R_iS_i^*,R_jS_j^*\rangle_{\mathrm{HS}},
 \label{eq:trace-contraction-Gram}
\end{equation}
so the contracted scalar kernel is of positive type.

\begin{definition}
For a finite graph $\mathfrak G$ on $X$, Schrijver's $\vartheta'$ is
\begin{equation}
 \vartheta'(\mathfrak G)=
 \max\left\{\sum_{x,y\in X}Z(x,y):
 \begin{array}{l}
 Z\in\RR^{X\times X},\\
 Z\succeq0,\quad Z(x,y)\geq0,\quad \displaystyle\sum_{x\in X}Z(x,x)=1,\\
 Z(x,y)=0\quad\text{whenever }x\sim_{\mathfrak G}y
 \end{array}\right\}.
 \label{eq:Schrijver-theta-prime}
\end{equation}
\end{definition}

Schrijver's bound satisfies
\begin{equation}
 \alpha(\mathfrak G)\leq\vartheta'(\mathfrak G)
 \label{eq:Schrijver-independence-bound}
\end{equation}
for every finite graph $\mathfrak G$~\cite{Schrijver1979} and specializes
to the linear programming bound in the Hamming space setting we consider.

\begin{theorem}[generalized moving-subspace bound]
\label{thm:generalized-moving-subspace-bound}
Let $0\ne\rho\succeq0$ in $\mathcal H_{E,\Omega}$ and $\Lambda\in\RR$ satisfy
\begin{equation}
 \Theta_\Omega(\rho)\succeq\Lambda\rho,
 \label{eq:Hecke-survival}
\end{equation}
and let $\tau<\Lambda$.  Write
\begin{equation}
 D_\rho=\dim\ran\rho,
 \qquad
 r_\rho=\rank K_\rho(o,o).
 \label{eq:moving-subspace-ranks}
\end{equation}
Then
\begin{equation}
 \vartheta'(\mathfrak G_t(\tau))\leq
 \frac{1-\tau}{\Lambda-\tau}\frac{D_\rho}{r_\rho}.
 \label{eq:generalized-moving-subspace-bound}
\end{equation}
\end{theorem}

\begin{proof}
For a scalar kernel $B$ let $\overline B$ denote its uniform average over
$X^2$.  Semidefinite duality and averaging over $\mathsf G$ show that any
invariant positive-type kernel $B$ with $\overline B>0$ and
$B(x,y)\leq0$ on allowed off-diagonal pairs satisfies
\begin{equation}
 \vartheta'(\mathfrak G)\leq\frac{B(o,o)}{\overline B}.
 \label{eq:invariant-theta-prime-dual-criterion}
\end{equation}

Since $\mathsf S_t(\rho)\succeq0$ belongs to $\mathcal H_E$ and
$z_\Omega$ is central,
\begin{equation}
 \mathsf S_t(\rho)-\Lambda\rho
 =(1-z_\Omega)\mathsf S_t(\rho)(1-z_\Omega)
  +\Theta_\Omega(\rho)-\Lambda\rho\succeq0.
 \label{eq:global-Hecke-positivity}
\end{equation}
The first summand is positive semidefinite since $\mathsf S_t(\rho)$ is, and the rest is positive semidefinite by the assumption~\eqref{eq:Hecke-survival}.
Define
\begin{equation}
 \mathcal Q_\rho(x,y)
 =\mathcal T_E(\rho,\rho)(x,y)
 =\|K_\rho(x,y)\|_{\mathrm{HS}}^2,
 \qquad
 B_{\rho,\tau}=(t-\tau)\mathcal Q_\rho.
 \label{eq:Hecke-certificate}
\end{equation}
The bilinear trace contraction shows that $\mathcal Q_\rho$ and
\begin{equation}
 (t-\Lambda)\mathcal Q_\rho
 =\mathcal T_E(\mathsf S_t(\rho)-\Lambda\rho,\rho)
 \label{eq:Hecke-positive-kernel}
\end{equation}
are positive-type scalar kernels, while $\mathcal Q_\rho(x,y)\geq0$
pointwise.  Thus
\[
 B_{\rho,\tau}
 =(t-\Lambda)\mathcal Q_\rho
  +(\Lambda-\tau)\mathcal Q_\rho
\]
is of positive type and is nonpositive when
$x\ne y$ and $t(x,y)\leq\tau$.  The kernel normalization gives
\begin{equation}
 \overline{\mathcal T_E(C,D)}
 =\frac1{|X|^2}\sum_{x,y\in X}
 \tr\bigl(K_C(x,y)K_D(y,x)\bigr)
 =\Tr(CD).
 \label{eq:trace-contraction-mean}
\end{equation}
Since
$B_{\rho,\tau}=\mathcal T_E(\mathsf S_t(\rho)-\tau\rho,\rho)$,
Equation~\eqref{eq:trace-contraction-mean} gives
\begin{align}
 \overline{B_{\rho,\tau}}
 &=\Tr\bigl((\mathsf S_t(\rho)-\tau\rho)\rho\bigr)\notag\\
 &=\Tr\bigl((\mathsf S_t(\rho)-\Lambda\rho)\rho\bigr)
   +(\Lambda-\tau)\Tr(\rho^2)\notag\\
 &\geq(\Lambda-\tau)\Tr(\rho^2)>0.
 \label{eq:Hecke-certificate-mean}
\end{align}
Here the last inequality uses
$\mathsf S_t(\rho)-\Lambda\rho\succeq0$ from
Equation~\eqref{eq:global-Hecke-positivity} and $\rho\succeq0$.
Therefore $B_{\rho,\tau}$ is feasible.

It remains to bound the objective.  We first normalize $\rho$.
In the induced-function model, let $p_\rho$ be the projection onto the
support of $K_\rho(o,o)$ and let $W_\rho=K_\rho(o,o)^{-1/2}$ on that
support and zero on its kernel.  Since $\mathsf S_t$ is scalar multiplication,
$\mathsf S_t(W_\rho \rho W_\rho) = W_\rho \mathsf S_t(\rho) W_\rho$.
Centrality of $z_\Omega$ then gives
\[
 \Theta_\Omega(W_\rho\rho W_\rho)
 =W_\rho\Theta_\Omega(\rho)W_\rho
 \succeq\Lambda W_\rho\rho W_\rho.
\]
Moreover,
\[
 K_{W_\rho\rho W_\rho}(o,o)=p_\rho,
 \qquad
 \rank K_{W_\rho\rho W_\rho}(o,o)=r_\rho.
\]
Positivity gives $\rho=p_\rho\rho p_\rho$, so $W_\rho$ is invertible on
$\ran\rho$ and
\[
 \rank(W_\rho\rho W_\rho)=\rank\rho=D_\rho.
\]
Replacing $\rho$ by $W_\rho\rho W_\rho$ and retaining the notation $\rho$,
we may therefore assume that
\[
 K_\rho(o,o)=p_\rho,
 \qquad \Tr(\rho)=r_\rho,
 \qquad \rank\rho=D_\rho.
\]

Apply the feasible-kernel construction above to this normalized datum and
write $B_\tau=B_{\rho,\tau}$.  Then
\[
 B_\tau(o,o)=(1-\tau)r_\rho,
 \qquad
 \overline{B_\tau}
 \geq(\Lambda-\tau)\Tr(\rho^2).
\]
Trace Cauchy--Schwarz gives
$\Tr(\rho^2)\geq r_\rho^2/D_\rho$.  Combining this with
Equation~\eqref{eq:invariant-theta-prime-dual-criterion} proves
Equation~\eqref{eq:generalized-moving-subspace-bound}.
\end{proof}

\subsection{Horn pairs and asymptotic bounds}

We first fix the Hamming homogeneous space.  Let
$\mathsf B_n=C_2^n\rtimes\mathfrak S_n$ and identify
\begin{equation}
 X_n=\mathsf B_n/\mathfrak S_n\cong\{0,1\}^n,
 \qquad
 t_n(x,y)=1-\frac{2d_H(x,y)}n.
 \label{eq:Hamming-homogeneous-space}
\end{equation}
Let $\mathsf B_n$ act on $\CC^n$ by signed permutations and put
\begin{equation}
 \ell_{n,x}=\frac1{\sqrt n}\sum_{i=1}^n(-1)^{x_i}e_i.
 \label{eq:Hamming-coordinate-lift}
\end{equation}
Then $\ell_{n,0}$ is $\mathfrak S_n$-fixed and
$t_n(x,y)=\langle\ell_{n,x},\ell_{n,y}\rangle$.  Thus $t_n$ is the
positive-type coordinate associated with $\ell_n$ in
Definition~\ref{def:threshold-graph}.

Fix $m$ and give $(\CC^m)^{\otimes n}$ its tensor-permutation $\mathfrak S_n$-action.
Fourier transform on $C_2^n$ gives
\begin{equation}
 \Ind_{\mathfrak S_n}^{\mathsf B_n}(\CC^m)^{\otimes n}
 \cong(\CC^2\otimes\CC^m)^{\otimes n}.
 \label{eq:Hamming-induced-Fourier-model}
\end{equation}
Here the $i$th generator of $C_2^n$ acts by $\zeta_i$, where
$\zeta=\operatorname{diag}(1,-1)$ on $\CC^2$, and $\mathfrak S_n$ permutes
the $n$ tensor factors.  Write
$\sigma=\left(\begin{smallmatrix}0&1\\1&0\end{smallmatrix}\right)$ and let
$\sigma_i$ act in the $i$th $\CC^2$ factor.  The multiplication-by-$t_n$
map is
\begin{equation}
 \mathsf S_{t_n}(A)=\frac1n\sum_{i=1}^n\sigma_iA\sigma_i.
 \label{eq:Hamming-multiplication-map}
\end{equation}

\begin{definition}
\label{def:Hamming-tensor-pair}
Let $K,L\in M_m(\CC)$ satisfy
$K,L\succeq0$ and $\tr(K+L)=1$.  In the decomposition
$\CC^2\otimes\CC^m\cong\CC^m\oplus\CC^m$, put
\begin{equation}
 \mathcal A_n=(K\oplus L)^{\otimes n},\qquad
 \mathcal R_n=\mathcal A_n^{1/2}.
 \label{eq:Hamming-test-operators}
\end{equation}
Since $K\oplus L$ is block diagonal, both operators belong to the generalized
Hecke algebra.
\end{definition}

\begin{lemma}
\label{lem:Hamming-trace-identity}
For $\mathcal A_n$ and $\mathcal R_n$ as above, we have
\begin{equation}
 \Tr(\mathcal R_n^2)=1,\qquad
 \Tr\bigl(\mathcal R_n\mathsf S_{t_n}(\mathcal R_n)\bigr)
 =2\tr(K^{1/2}L^{1/2}).
 \label{eq:Hamming-exact-overlap}
\end{equation}
\end{lemma}

\begin{proof}
The trace of $\mathcal A_n$ is $\tr(K+L)^n=1$.  Conjugation by
$\sigma_i$ exchanges $K^{1/2}$ and $L^{1/2}$ in the $i$th block.  That factor
therefore contributes $2\tr(K^{1/2}L^{1/2})$, while each remaining factor
contributes $\tr(K+L)=1$.  Averaging over $i$ proves the displayed formula.
\end{proof}

\begin{definition}[normalized Horn pair]
\label{def:normalized-Horn-pair}
A pair $(K,L)$ satisfying Definition~\ref{def:Hamming-tensor-pair} is a
\emph{normalized Horn pair}.
Write $G=K+L$, and let $k,l,g\in\RR_{\geq0}^m$ be the decreasing eigenvalue
lists of $K,L,G$, padded by zeros when necessary.  Following \cite{GayJeronimoLiu2026}, we define
\begin{equation}
 \Gamma(K,L)=2\tr(K^{1/2}L^{1/2}),\qquad
 \Phi(K,L)=\Ent(K)+\Ent(L)-\Ent(G).
 \label{eq:local-channel-coordinates}
\end{equation}
Here $\Ent(A)=-\tr(A\log_2A)$, with the usual continuous value at zero.
\end{definition}
Gay, Jeronimo, and Liu organize their paper around the analogy to the Horn problem.
To show that it is a bound for codes, they identify $\Phi$ with the uniform-prior Holevo
information and $(1-\Gamma)/2$ with the pretty-good-measurement bit error
\cite[Proposition~8.1]{GayJeronimoLiu2026} and use the bound in \cite{AlrabiahGuruswami2026}.
They then use $\Phi$ as the objective with constraint on $\Gamma$ in their matrix hierarchy
\cite[Definition~8.2 and Theorem~8.3]{GayJeronimoLiu2026}.

To decompose the tensor powers, let
$W_\lambda^{(m)}$ denote the irreducible polynomial $U(m)$-module of highest
weight $\lambda$, and let $S^\lambda$ denote the Specht module of shape
$\lambda$.  The tensor model has the decompositions
\begin{align}
 (\CC^2\otimes\CC^m)^{\otimes n}
 &\cong
 \bigoplus_{|\alpha|+|\beta|=n}
 U_{\alpha,\beta}\otimes
 W_\alpha^{(m)}\otimes W_\beta^{(m)},
 \label{eq:wreath-Schur-Weyl-decomposition}\\
 (\CC^m)^{\otimes n}
 &\cong
 \bigoplus_{\gamma\vdash n}S^\gamma\otimes W_\gamma^{(m)}.
 \label{eq:fiber-Schur-Weyl-decomposition}
\end{align}
Here $U_{\alpha,\beta}$ is the irreducible $\mathsf B_n$-module indexed by
the bipartition $(\alpha,\beta)$; only partitions of length at most $m$
occur.  The passage between the two decompositions is Littlewood--Richardson
branching:
\begin{equation}
 \Res_{\mathfrak S_n}^{\mathsf B_n}U_{\alpha,\beta}
 \cong
 \Ind_{\mathfrak S_{|\alpha|}\times\mathfrak S_{|\beta|}}^{\mathfrak S_n}
       (S^\alpha\boxtimes S^\beta)
 \cong
 \bigoplus_{\gamma\vdash n}c_{\alpha\beta}^{\gamma}S^\gamma.
 \label{eq:wreath-Littlewood-Richardson-branching}
\end{equation}
Heuristically, this branching rule quantizes the moment map
$\mu_+:\mathfrak u(m)^*\times\mathfrak u(m)^*\to\mathfrak u(m)^*$,
$\mu_+(K,L)=K+L$.  Under the orbit-method correspondence, the highest weight
$\lambda$ parametrizes the coadjoint orbit quantized by $W_\lambda^{(m)}$, where
at scale $n$, $\lambda/n$ is its normalized classical parameter.  Thus
$(\alpha/n,\beta/n,\gamma/n)$ should approximate the spectral data $(k,l,g)$.
At finite $n$, the compatible triples are those for which
$c_{\alpha\beta}^{\gamma}>0$.
The spectral input we will use is the following concentration theorem.
A similar concentration result was also
studied earlier by Alicki, Rudnicki, and Sadowski
\cite{AlickiRudnickiSadowski1988}.

\begin{theorem}[Keyl--Werner]
\label{thm:Keyl-Werner}
Let $\tau$ be a density matrix on $\CC^m$ with decreasing eigenvalue list
$s$, and let $\Pi_\lambda^{(N)}$ be the $\mathfrak S_N$-isotypic projection
of shape $\lambda$ in $(\CC^m)^{\otimes N}$.  The numbers
\begin{equation}
 \mathbb P_{\tau,N}(\lambda)
 =\tr\bigl(\tau^{\otimes N}\Pi_\lambda^{(N)}\bigr)
 \label{eq:Schur-Weyl-measure}
\end{equation}
form a probability measure on partitions of $N$ with at most $m$ parts.
View $\lambda/N$ as an $m$-vector by appending zero parts.  For every
$\varepsilon>0$,
\begin{equation}
 \sum_{\substack{\lambda\vdash N\\
        \|\lambda/N-s\|_1>\varepsilon}}
 \tr\bigl(\tau^{\otimes N}\Pi_\lambda^{(N)}\bigr)
 \longrightarrow0.
 \label{eq:Keyl-Werner-concentration}
\end{equation}
Equivalently, the normalized row-length vector $\lambda/N$ converges in
probability to the spectrum $s$.
\end{theorem}

This is the form of Keyl and Werner's theorem that we will use below
\cite[Section~II, Theorem]{KeylWerner2001}.
It follows from their large-deviation principle because the LDP's good rate
function vanishes only at $s$, and therefore has a strictly positive minimum
on the closed subset of the probability simplex on which
$\|u-s\|_1\geq\varepsilon$.

For $\varepsilon>0$, let $z_{n,\varepsilon}$ be the central projection onto
the $\mathsf B_n$-types satisfying
\begin{equation}
 \left\|\left(\frac\alpha n,\frac\beta n\right)-(k,l)\right\|_1
 \leq\varepsilon,
 \label{eq:output-typical-window}
\end{equation}
and let $q_{n,\varepsilon}$ be the sum of the $\mathfrak S_n$-isotypic
projections in $(\CC^m)^{\otimes n}$ satisfying
\begin{equation}
 \left\|\frac\gamma n-g\right\|_1\leq\varepsilon.
 \label{eq:fiber-typical-window}
\end{equation}

\begin{lemma}
\label{lem:joint-spectral-concentration}
For every $\varepsilon>0$,
\begin{equation}
 \Tr\bigl(\mathcal A_n(I-z_{n,\varepsilon})\bigr)\longrightarrow0,
 \qquad
 \tr\bigl(G^{\otimes n}(I-q_{n,\varepsilon})\bigr)\longrightarrow0.
 \label{eq:fixed-window-concentration}
\end{equation}
\end{lemma}

\begin{proof}
The second limit is Theorem~\ref{thm:Keyl-Werner} applied to $G$.  For the
first, put $p=\tr K$ and $q=\tr L=1-p$.  Under the state
$\mathcal A_n=(K\oplus L)^{\otimes n}$, the number $N_K$ of $K$-labels has
distribution $\operatorname{Bin}(n,p)$, and hence
\begin{equation}
 \Pr\left\{\left|\frac{N_K}{n}-p\right|>\eta\right\}
 \leq\frac{p(1-p)}{n\eta^2}.
 \label{eq:block-count-Chebyshev}
\end{equation}
Given $N_K=a$, the two Young diagrams are independent with laws
$\mathbb P_{K/p,a}$ and $\mathbb P_{L/q,n-a}$.  A factor is omitted when its
block mass is zero; a normalized diagram at tensor size zero may be assigned
an arbitrary value.  If $p>0$, then $N_K\to\infty$ in probability, so
Theorem~\ref{thm:Keyl-Werner}, applied at the random tensor size $N_K$, gives
\[
 \left\|\frac\alpha{N_K}-\frac{k}{p}\right\|_1
 \longrightarrow0
 \quad\text{in probability};
\]
indeed, one first restricts to $N_K\geq A$ and then lets $A\to\infty$.
The analogous statement holds for $\beta$ when $q>0$.  Finally,
\begin{align*}
 \left\|\left(\frac\alpha n,\frac\beta n\right)-(k,l)\right\|_1
 &\leq \frac{N_K}{n}
   \left\|\frac\alpha{N_K}-\frac{k}{p}\right\|_1
  +\frac{n-N_K}{n}
   \left\|\frac\beta{n-N_K}-\frac{l}{q}\right\|_1
  +2\left|\frac{N_K}{n}-p\right|,
\end{align*}
with the zero-mass terms omitted.  Equation~\eqref{eq:block-count-Chebyshev}
and the two conditional limits prove the first assertion.
\end{proof}

Choose integers $1=N_1<N_2<\cdots$ so that, for $n\geq N_j$, both errors in
\eqref{eq:fixed-window-concentration} with $\varepsilon=1/j$ are at most
$1/j$, and set $a_n=1/j$ for $N_j\leq n<N_{j+1}$.  Then $a_n\downarrow0$
and both errors remain $o(1)$ with $\varepsilon=a_n$.  Put
$\mathsf h(a)=-\sum_i a_i\log_2a_i$ for a nonnegative vector whose entries
need not sum to one.

\begin{definition}[joint typical subspace]
\label{def:joint-typical-subspace}
For the operators associated with $(K,L)$, put
$z_n=z_{n,a_n}$ and $q_n=q_{n,a_n}$.
Since $q_n$ is $\mathfrak S_n$-equivariant, it defines the pointwise
multiplier
\begin{equation}
 (\widehat q_nf)(b\mathfrak S_n)
 =bq_nb^{-1}f(b\mathfrak S_n).
 \label{eq:induced-fiber-projection}
\end{equation}
Put $E_n=q_n(\CC^m)^{\otimes n}$ and
\begin{equation}
 p_n=z_n\widehat q_n.
 \label{eq:joint-typical-subspace}
\end{equation}
The range of $p_n$ is the \emph{joint typical subspace}.  Equivalently,
it is the $z_n$-isotypic summand of
$\Ind_{\mathfrak S_n}^{\mathsf B_n}E_n$, represented inside the full Fourier
model.  Let $\Omega_n$ denote its set of $\mathsf B_n$-types and let
$\Theta_n$ be the map induced by $\mathsf S_{t_n}$ on this summand.
\end{definition}

\begin{proposition}
\label{prop:joint-Schur-Weyl-concentration}
The trace of $\mathcal A_n$ is concentrated on the joint typical subspace:
\begin{equation}
 \Tr\bigl(\mathcal A_n(I-p_n)\bigr)=o(1).
 \label{eq:joint-subspace-mass}
\end{equation}
Moreover, for every sequence of nonzero positive Hecke operators $\rho_n$
in the generalized Hecke algebra satisfying $\rho_n=p_n\rho_np_n$,
\begin{align}
 \limsup_{n\to\infty}\frac1n\log_2D_{\rho_n}
 &\leq\Ent(K)+\Ent(L),
 \label{eq:typical-output-rank}\\
 \liminf_{n\to\infty}\frac1n\log_2r_{\rho_n}
 &\geq\Ent(G).
 \label{eq:typical-basepoint-rank}
\end{align}
Consequently,
\begin{equation}
 \limsup_{n\to\infty}\frac1n\log_2
 \frac{D_{\rho_n}}{r_{\rho_n}}
 \leq\Phi(K,L).
 \label{eq:typical-cost}
\end{equation}
\end{proposition}

\begin{proof}
\emph{Mass.}
The choice of $a_n$ following Lemma~\ref{lem:joint-spectral-concentration}
gives
\begin{equation}
 \Tr(\mathcal A_n(I-z_n))=o(1),
 \label{eq:global-Schur-Weyl-mass}
\end{equation}
and
\begin{equation}
 \tr(G^{\otimes n}(I-q_n))=o(1).
 \label{eq:fiber-Schur-Weyl-mass}
\end{equation}
By the definition of the induced pointwise multiplier, the left side of
\eqref{eq:fiber-Schur-Weyl-mass} is
$\Tr(\mathcal A_n(I-\widehat q_n))$.  Since $z_n$ and $\widehat q_n$
commute,
\[
 I-p_n=I-z_n\widehat q_n
 \preceq (I-z_n)+(I-\widehat q_n).
\]
Taking the trace against $\mathcal A_n$ and using the preceding two estimates
proves \eqref{eq:joint-subspace-mass}.

\emph{The bound on $D_{\rho_n}$.}
If $\rho_n=p_n\rho_np_n$, then
$\ran\rho_n\subseteq\ran p_n\subseteq\ran z_n$.  The decomposition
\eqref{eq:wreath-Schur-Weyl-decomposition} consequently gives
\begin{equation}
 D_{\rho_n}\leq\rank z_n
 =\sum_{\substack{|\alpha|+|\beta|=n\\
  \|(\alpha/n,\beta/n)-(k,l)\|_1\leq a_n}}
   (\dim U_{\alpha,\beta})
   (\dim W_\alpha^{(m)})(\dim W_\beta^{(m)}).
 \label{eq:global-typical-rank}
\end{equation}
Here
\begin{equation}
 \dim U_{\alpha,\beta}
 =\binom n{|\alpha|}f^\alpha f^\beta,
 \label{eq:wreath-irrep-dimension}
\end{equation}
where $f^\lambda=\dim S^\lambda$.  For fixed $m$,
the Weyl dimension formula shows that $\dim W_\alpha^{(m)}$ and
$\dim W_\beta^{(m)}$ are polynomially bounded in $n$; the number of
bipartitions in the sum is polynomially bounded as well.  Thus these factors
contribute only $2^{o(n)}$.
The hook formula gives, uniformly over partitions of length at most $m$,
\begin{equation}
 \frac1n\log_2\!\left(
  \binom n{|\alpha|}f^\alpha f^\beta\right)
 =\mathsf h(\alpha/n)+\mathsf h(\beta/n)+o(1).
 \label{eq:wreath-hook-exponent}
\end{equation}
On the window \eqref{eq:output-typical-window} with $\varepsilon=a_n$, the
right side converges uniformly to $\mathsf h(k)+\mathsf h(l)$.  Applying this
estimate to
\eqref{eq:global-typical-rank} proves \eqref{eq:typical-output-rank}.

\emph{The bound on $r_{\rho_n}$.}
Let $B_n=K_{\rho_n}(o,o)$.  Since $\rho_n$ is nonzero and positive,
$B_n$ is a nonzero positive operator.  It is $\mathfrak S_n$-equivariant,
and the relation $\rho_n=p_n\rho_np_n$ implies
$B_n=q_nB_nq_n$.  Hence $\ran B_n$ is a nonzero
$\mathfrak S_n$-submodule of
\[
 E_n=\bigoplus_{\substack{\gamma\vdash n\\
  \|\gamma/n-g\|_1\leq a_n}}
 S^\gamma\otimes W_\gamma^{(m)}.
\]
It therefore contains a copy of $S^\gamma$ for at least one partition in the
fiber window, and so $r_{\rho_n}=\rank B_n\geq f^\gamma$.  Uniformly on that
window, the hook formula gives
\[
 \frac1n\log_2f^\gamma=\mathsf h(\gamma/n)+o(1)
 =\mathsf h(g)+o(1),
\]
which proves \eqref{eq:typical-basepoint-rank}.  Finally,
$\mathsf h(k)=\Ent(K)$, $\mathsf h(l)=\Ent(L)$, and
$\mathsf h(g)=\Ent(G)$.  Combining \eqref{eq:typical-output-rank} and
\eqref{eq:typical-basepoint-rank} proves \eqref{eq:typical-cost}.
\end{proof}

\begin{definition}
\label{def:asymptotic-realization}
An \emph{asymptotic realization} of a normalized Horn pair
$(K,L)$ is a sequence
\begin{equation}
 (E_n,\Omega_n,\rho_n,\Lambda_n)
 \label{eq:asymptotic-realization}
\end{equation}
such that $0\ne\rho_n\succeq0$ in $\mathcal H_{E_n,\Omega_n}$ satisfies
$\Theta_n(\rho_n)\succeq\Lambda_n\rho_n$ and
\begin{align}
 \liminf_{n\to\infty}\Lambda_n
 &\geq\Gamma(K,L),
 \label{eq:asymptotic-Hecke-survival}\\
 \limsup_{n\to\infty}\frac1n\log_2\frac{D_{\rho_n}}{r_{\rho_n}}
 &\leq\Phi(K,L).
 \label{eq:asymptotic-Hecke-cost}
\end{align}
\end{definition}

The following theorem shows that the optimization problem for Horn pairs
is the limit of the generalized moving-subspace construction in the above sense.

\begin{theorem}
\label{thm:realization}
Every normalized Horn pair has an asymptotic realization.  In
particular, if
\begin{equation}
 \Gamma(K,L)>1-2\delta,
 \label{eq:strict-channel-threshold}
\end{equation}
then
\begin{equation}
 R_D(\delta)\leq\Phi(K,L).
 \label{eq:channel-bound}
\end{equation}
\end{theorem}

\begin{proof}
Proposition~\ref{prop:joint-Schur-Weyl-concentration} gives
$\eta_n=\Tr(\mathcal A_n(I-p_n))=o(1)$.  Put
$\mathcal R'_n=p_n\mathcal R_np_n$.  Since
\[
 \mathcal R_n-\mathcal R'_n
 =(I-p_n)\mathcal R_n+p_n\mathcal R_n(I-p_n)
\]
and left multiplication by $p_n$ is a contraction in Hilbert--Schmidt norm,
\[
 \|(I-p_n)\mathcal R_n\|_2^2
 =\|\mathcal R_n(I-p_n)\|_2^2
 =\Tr(\mathcal A_n(I-p_n))=\eta_n.
\]
Thus $\|\mathcal R_n-\mathcal R'_n\|_2\leq2\sqrt{\eta_n}=o(1)$.
Because $\mathcal R'_n=p_n\mathcal R'_np_n$, the definition of
$\Theta_n$ gives
\[
 \Tr\bigl(\mathcal R'_n\Theta_n(\mathcal R'_n)\bigr)
 =
 \Tr\bigl(\mathcal R'_n\mathsf S_{t_n}(\mathcal R'_n)\bigr).
\]
Since $\mathsf S_{t_n}$ is an average of conjugations by unitaries $\sigma_i$ \eqref{eq:Hamming-multiplication-map}, it is a contraction for the Hilbert--Schmidt norm.  As $\|\mathcal R_n\|_2=1$, Cauchy--Schwarz gives
\begin{align*}
 \left|\Tr((\mathcal R'_n)^2)-1\right|
 &\leq
 \|\mathcal R'_n-\mathcal R_n\|_2
 \bigl(\|\mathcal R'_n\|_2+\|\mathcal R_n\|_2\bigr)=o(1),\\
 \left|
 \Tr\bigl(\mathcal R'_n\Theta_n(\mathcal R'_n)\bigr)
 -
 \Tr\bigl(\mathcal R_n\mathsf S_{t_n}(\mathcal R_n)\bigr)
 \right|
 &\leq
 \|\mathcal R'_n-\mathcal R_n\|_2
 \bigl(\|\mathcal R'_n\|_2+\|\mathcal R_n\|_2\bigr)=o(1).
\end{align*}
Equation~\eqref{eq:Hamming-exact-overlap} therefore yields
\begin{equation}
 \frac{\Tr(\mathcal R'_n\Theta_n(\mathcal R'_n))}
      {\Tr((\mathcal R'_n)^2)}
 =\Gamma(K,L)+o(1),
 \label{eq:typical-survival}
\end{equation}
where $\Theta_n$ is the map induced by $\mathsf S_{t_n}$ on the isotypic
submodule~\eqref{eq:joint-typical-subspace}.

The map $\Theta_n$ is completely positive, so it preserves the cone of
positive semidefinite operators, and it is self-adjoint for the
Hilbert--Schmidt inner product.  Perron--Frobenius theory for positive maps
therefore gives a nonzero $\rho_n\succeq0$ such that
\[
 \Theta_n(\rho_n)=\Lambda_n\rho_n,
 \qquad
 \Lambda_n=r(\Theta_n).
\]
In particular, $\Lambda_n$ is a nonnegative eigenvalue.  Since $\Theta_n$ is
self-adjoint, all its eigenvalues are real; as each has absolute value at most
$r(\Theta_n)$, it follows that $\Lambda_n$ is its largest eigenvalue.
Consequently, Equation~\eqref{eq:typical-survival} and the variational
characterization of the largest eigenvalue give, for some
$\epsilon_n\downarrow0$,
\begin{equation}
 \Theta_n(\rho_n)=\Lambda_n\rho_n,
 \qquad
 \Lambda_n\geq\Gamma(K,L)-\epsilon_n.
 \label{eq:typical-Perron-survival}
\end{equation}

Since $\rho_n=p_n\rho_np_n$, Proposition~\ref{prop:joint-Schur-Weyl-concentration} gives \eqref{eq:asymptotic-Hecke-cost}.  Hence
$(E_n,\Omega_n,\rho_n,\Lambda_n)$ is an asymptotic realization
of $(K,L)$.

Finally set $d_n=\lceil\delta n\rceil$ and
$\tau_n=1-2d_n/n$.  Under~\eqref{eq:strict-channel-threshold},
$\Lambda_n-\tau_n$ is bounded below by a positive constant.
Theorem~\ref{thm:generalized-moving-subspace-bound} supplies an ordinary radial
Krawtchouk dual certificate with objective at most a constant multiple of
$D_{\rho_n}/r_{\rho_n}$.  Passing to the rate proves
\eqref{eq:channel-bound}.
\end{proof}

\section{The half-disk uncertainty principle}
\label{sec:half-disk}

Much of the recent progress on the Cohn--Elkies problem has originated from
the relationship to the $(-1)$-eigenfunction uncertainty principle, which
arises almost immediately from the formulation of Cohn--Elkies but whose true importance was only realized much later
\cite{CohnElkies2003,CohnGoncalves2019,OpenAI2026}.  Use the self-dual Fourier normalization
with kernel $e^{-2\pi\iu\langle x,y\rangle}$ and, after radialization, let
$f$ be a feasible solution to the Cohn--Elkies problem.
Rescale $F(r)=f(\alpha r)$ so that $F(0)=\widehat F(0)$.  With this
normalization, the resulting Cohn--Elkies upper bound on the density of
centers is $1$; let $R$ be the corresponding minimum distance between sphere
centers.  Then $g=\widehat F-F$ is a $(-1)$-eigenfunction of the Fourier
transform and $g(r)\geq0$ for $r\geq R$.  Conversely, beginning with such an
eigenfunction, we may try to construct a $(+1)$-eigenfunction with certain
properties and thus recover the original $F$.  This eigenfunction approach
underlies both numerical computation via Hermite polynomials and
exact solutions, such as the magic functions in dimensions $8$ and $24$
\cite{Viazovska2017,CohnKumarMillerRadchenkoViazovska2017}.

The Delsarte program has no analogue of the rescaling trick used to connect the
linear program with an eigenfunction uncertainty principle.  The philosophical motivation for us is the observation that at rate $1/2$,
feasible dual solutions of Delsarte's linear program give Krawtchouk eigenfunctions.
One informal slogan we could give is the Delsarte linear program at rate $1/2$ behaves
as though it is dual to weight enumerators for self-dual codes,
just as the Cohn-Elkies linear program behaves as though it is dual to the theta
functions of unimodular lattices.
With Krawtchouk eigenfunctions, we can formulate a mass-concentration statement for high-dimensional Krawtchouk eigenfunctions.  Proving the concentration theorem derives the lower
bound for both Krawtchouk sign uncertainty problems.  In particular, the
case of the $(-1)$-uncertainty problem gives a lower bound for the Delsarte
bound as well.  Finally, Theorem~\ref{thm:realization} turns the same lower
bound into the matrix inequality in Corollary~\ref{cor:half-rate-matrix-inequality}.

\subsection{The analytic core}

Set $w_i=\binom ni$ and define
\begin{equation}
 (\cK_nu)_k=2^{-n/2}\sum_{i=0}^nK_i^{(n)}(k)u_i.
 \label{eq:Krawtchouk-transform}
\end{equation}
Krawtchouk orthogonality makes $\cK_n$ an orthogonal involution on
$\RR^{n+1}$ for the weighted inner product
$\langle u,v\rangle_w=\sum_iw_iu_iv_i$.  Let $\cE_n^\pm$ be its eigenspaces
with eigenvalues $\pm1$.

\begin{definition}
For a radial vector $u=(u_0,\ldots,u_n)$, let
\begin{equation}
 r_n(u)=\min\{r\in\{0,\ldots,n\}:u_i\geq0
 \text{ for every }i\geq r\},
 \label{eq:Krawtchouk-last-sign-radius}
\end{equation}
with value $+\infty$ if the set is empty.  For
$\varsigma\in\{-1,+1\}$, define
\begin{equation}
 A^{\mathrm K}_{\varsigma}(n)
 =\inf\{r_n(u):0\ne u\in\cE_n^\varsigma,\ u_0=0\},
 \label{eq:Krawtchouk-sign-constant}
\end{equation}
again with value $+\infty$ if there is no such vector with finite last-sign
radius.
\end{definition}

Put
\begin{equation}
 G_u(z)=\sum_{i=0}^nw_iu_iz^i,
 \qquad \omega(z)=\frac{1-z}{1+z}.
 \label{eq:weighted-generating-polynomial}
\end{equation}

\begin{lemma}[Krawtchouk transform functional equation]
\label{lem:Krawtchouk-functional-equation}
For every radial vector $u$,
\begin{equation}
 G_{\cK_nu}(z)=2^{-n/2}(1+z)^nG_u(\omega(z)).
 \label{eq:Krawtchouk-functional-equation}
\end{equation}
Thus, if $g\in\cE_n^\varepsilon$ with $\varepsilon\in\{-1,1\}$, then
\begin{equation}
 G_g(z)=\varepsilon2^{-n/2}(1+z)^nG_g(\omega(z)).
 \label{eq:Krawtchouk-eigenfunction-functional-equation}
\end{equation}
\end{lemma}

\begin{proof}
The standard generating function and the hypergeometric formula underlying
its self-duality are recorded in
\cite[Eqs.~(18.23.3) and (18.20.6)]{NISTHandbook2010}.  In our normalization,
use Equation~\eqref{eq:Krawtchouk-generating-function} and the weighted
Krawtchouk self-duality relation
$w_iK_k^{(n)}(i)=w_kK_i^{(n)}(k)$, then interchange the two finite sums.
\end{proof}

Normalize a nonzero eigenfunction by
$\sum_iw_i|g_i|=1$ and write $P=G_g$.  On the unit circle $|P|\leq1$.
For $-1\leq y\leq1$, the point $\omega(iy)$ lies on that circle, so the
functional equation gives
\begin{equation}
 |P(iy)|\leq \exp\bigl(n\varphi(y)\bigr),
 \qquad
 \varphi(y)=\frac12\log\frac{1+y^2}{2}.
 \label{eq:diameter-bound}
\end{equation}
The polynomial is therefore controlled on the two pieces of the boundary of
each half-disk.  The corresponding Dirichlet problem then yields the constant
$\delta_\pi$.

The conformal map
\begin{equation}
 W(z)=\left(\frac{z-\iu}{z+\iu}\right)^2
 \label{eq:half-disk-map}
\end{equation}
sends the right half-disk to the upper half-plane, the semicircular arc to
the negative real axis, and the imaginary diameter to the positive real
axis.  Put
\[
 \psi(x)=\frac12\log(1+x)-\log(1+\sqrt x).
\]
For $|z|<1$ and $\operatorname{Re}z>0$, define
\begin{equation}
 \mathcal B(z)=\operatorname{Re}\left[
 -\frac{\iu}{\pi}\int_0^\infty\frac{\psi(x)}{x-W(z)}\,dx
 \right],
 \label{eq:holomorphic-half-disk-extension}
\end{equation}
and define $\mathcal B$ on the left half-disk by
$\mathcal B(-z)=\mathcal B(z)$.  This is the unique continuous function on
the closed disk that is harmonic in each open half-disk, equals zero on the
unit circle, and equals $\varphi(y)$ at $z=\iu y$.

Define $D:(0,\infty)\to\mathbb R$ by
\begin{equation}
 D(a)=
 \begin{cases}
 \displaystyle\frac12-\frac{2a\log a}{\pi(a^2-1)},&a\ne1,\\[6pt]
 \displaystyle\frac12-\frac1\pi=\delta_\pi,&a=1.
 \end{cases}
 \label{eq:D-a}
\end{equation}

\begin{lemma}[Half-disk Dirichlet-to-Neumann profile]
\label{lem:half-disk}
There is a continuous function $q$ on the unit circle such that, uniformly
in $\theta$,
\begin{equation}
 \frac{\mathcal B(re^{\iu\theta})}{\log r}
 \longrightarrow q(e^{\iu\theta})
 \qquad(r\uparrow1).
 \label{eq:normal-profile}
\end{equation}
Because $\log r\sim r-1$, the function $q$ is the outward normal derivative
of $\mathcal B$ along the circular boundary.
On the open right semicircle, write
\begin{equation}
 a=\tan\!\left(\frac\pi4-\frac\theta2\right),
 \qquad -\frac\pi2<\theta<\frac\pi2.
\end{equation}
Then
\begin{equation}
 q(e^{\iu\theta})=D(a).
\end{equation}
The function $D$ has its
unique minimum at $a=1$, with $D(1)=\delta_\pi$.  At the two corners
$q(\pm\iu)=1/2$, and $q(-z)=q(z)$.  Consequently,
\begin{equation}
 \min_{|z|=1}q(z)=\delta_\pi,
 \label{eq:normal-lower}
\end{equation}
with equality exactly at $z=\pm1$.
\end{lemma}

\begin{proof}
Taking real parts in
Equation~\eqref{eq:holomorphic-half-disk-extension} gives the Poisson formula
\begin{equation}
 \mathcal B(z)=\frac1\pi\int_0^\infty
 \frac{\operatorname{Im}W(z)\,\psi(x)}{|x-W(z)|^2}\,dx.
 \label{eq:Poisson}
\end{equation}
On the circular arc, $W$ has negative real boundary values, so the boundary
value in Equation~\eqref{eq:Poisson} is zero.  On the imaginary diameter,
\[
 W(\iu y)=\left(\frac{1-y}{1+y}\right)^2,
 \qquad
 \psi(W(\iu y))=\varphi(y),
\]
which verifies the stated Dirichlet data.
For $z=re^{\iu\theta}$,
\[
 W(e^{\iu\theta})=-a^2,
 \qquad
 \frac{\operatorname{Im}W(re^{\iu\theta})}{1-r}
 \longrightarrow a(1+a^2).
\]
Integration by parts followed by $x=t^2$ gives
\begin{align}
 \int_0^\infty\frac{\psi(x)}{(x+a^2)^2}\,dx
 &=\int_0^\infty
 \frac{t-1}{(t+1)(t^2+1)(t^2+a^2)}\,dt \notag\\
 &=\frac{2\log a}{(a^2-1)(a^2+1)}
   -\frac\pi{2a(a^2+1)}.
 \label{eq:rational-integral}
\end{align}
Since $\log r\sim-(1-r)$, Equations~\eqref{eq:Poisson}--
\eqref{eq:rational-integral} give
\[
 \lim_{r\uparrow1}\frac{\mathcal B(re^{\iu\theta})}{\log r}
 =-\frac{a(1+a^2)}\pi
   \int_0^\infty\frac{\psi(x)}{(x+a^2)^2}\,dx
 =D(a).
\]
Writing $a=e^x$ gives
\begin{equation}
 D(e^x)=\frac12-\frac1\pi\frac{x}{\sinh x}
 =\delta_\pi+\frac{x^2}{6\pi}+O(x^4).
 \label{eq:logistic-well}
\end{equation}
Since $x/\sinh x\leq1$, with equality only at $x=0$, the function $D$
has its unique minimum at $a=1$, with value $\delta_\pi$.

The convergence in \eqref{eq:normal-profile} is uniform away from the two corners.
The remainder of the proof will handle these two corners.
Consider the upper corner $\iu$ and use the local coordinates
$s=\operatorname{Re}z$ and $t=1-\operatorname{Im}z$.  On the diameter,
\[
 \varphi(1-t)=-t/2+O(t^2),
\]
whereas on the circular arc
$t=s^2/2+O(s^4)$.  Hence the harmonic remainder
\[
 U(z)=\mathcal B(z)+t/2
\]
has boundary values $O(|z-\iu|^2)$ on each of the two incident arcs.
Moreover,
\[
 |W(z)|\asymp |z-\iu|^2,
 \qquad
 |W'(z)|\ll |z-\iu|.
\]
Thus the boundary trace of
$\widetilde U=U\circ W^{-1}$ is piecewise continuously differentiable near
the origin and is $O(|x|)$ on both sides of the real axis.  Differentiating
its upper-half-plane Poisson integral and splitting the integral into
$|x|\leq2|w|$, $2|w|<|x|<1$, and $|x|\geq1$ gives
\[
 |\nabla\widetilde U(w)|
 \ll 1+|\log|w||.
\]
The chain rule therefore yields
\[
 |\nabla U(z)|
 \ll |z-\iu|\bigl(1+|\log|z-\iu||\bigr)=o(1)
\]
as $z\to\iu$ within the half-disk.  Since
$\mathcal B(e^{\iu\theta})=0$, integration along radial segments gives
\[
 \mathcal B(re^{\iu\theta})
 =-\int_r^1\left(\partial_\rho U(\rho e^{\iu\theta})
 +\frac{\sin\theta}{2}\right)\,d\rho.
\]
It follows uniformly as $(r,\theta)\to(1,\pi/2)$ that
\[
 \frac{\mathcal B(re^{\iu\theta})}{\log r}
 =\frac{1-r}{-\log r}\left(\frac{\sin\theta}{2}+o(1)\right)
 \longrightarrow\frac12.
\]
The same argument at $-\iu$, using $1/W$ and
$t=1+\operatorname{Im}z$, gives the other corner.
Together with the convergence away from the corners, this proves the uniform
convergence in Equation~\eqref{eq:normal-profile}.  Reflection handles the
left half-disk and gives $q(-z)=q(z)$.
\end{proof}

\begin{theorem}[Krawtchouk mass concentration]
\label{thm:mass}
For every fixed $c<\delta_\pi$, there are $C_c,\gamma_c>0$ such that every
$g\in\cE_n^+\cup\cE_n^-$ satisfies
\begin{equation}
 \sum_{0\leq i<\lceil cn\rceil}w_i|g_i|
 \leq C_ce^{-\gamma_cn}\sum_{i=0}^nw_i|g_i|.
 \label{eq:mass}
\end{equation}
\end{theorem}

\begin{proof}
Normalize the total weighted mass to one and put $P=G_g$.  The function
$\log|P|-n\mathcal B$ is subharmonic in either half-disk and nonpositive on both
boundary pieces by $|P|\leq1$ and Equation~\eqref{eq:diameter-bound}.
The maximum principle gives $\log|P(z)|\leq n\mathcal B(z)$ for $|z|<1$.

Choose $\eta>0$ with $c+2\eta<\delta_\pi$.  Lemma~\ref{lem:half-disk}
supplies a fixed $r<1$, close enough to one, for which
\[
 \mathcal B(re^{\iu\theta})\leq(c+\eta)\log r
\]
uniformly in $\theta$.  If $P(z)=\sum_ip_iz^i$, Cauchy's formula gives
$|p_i|\leq r^{\eta n}$ for $i<\lceil cn\rceil$.  Since $p_i=w_ig_i$, the
sum in Equation~\eqref{eq:mass} is at most $(n+1)r^{\eta n}$.  Because
$r^\eta<1$, this is bounded by $C_ce^{-\gamma_cn}$ for some constants
$C_c,\gamma_c>0$ depending only on $c$: the polynomial factor $n+1$ is
absorbed by taking a slightly weaker exponential rate.
\end{proof}

\subsection{The Krawtchouk sign uncertainty principle}

The mass theorem gives the lower bound for the two
sign-uncertainty problems, and it is the Hamming analogue of the ``mass-concentration'' mechanism used in \cite[Chapter~1]{OpenAI2026}.
The general strategy of noting that such mass-concentration statements are sufficient is as old as the sign uncertainty principle itself, being used in \cite{BourgainClozelKahane2010,GoncalvesOliveiraeSilvaSteinerberger2017}.
However, the pointwise estimates used in those papers were only able to recover a weaker asymptotic bound.

\begin{proposition}
\label{prop:Krawtchouk-sign-lower}
For each $\varsigma\in\{-1,+1\}$,
\begin{equation}
 \liminf_{n\to\infty}\frac{A^{\mathrm K}_{\varsigma}(n)}n
 \geq\delta_\pi.
 \label{eq:Krawtchouk-sign-lower}
\end{equation}
\end{proposition}

\begin{proof}
If $u\in\cE_n^\varsigma$ and $u_0=0$, evaluating
Equation~\eqref{eq:Krawtchouk-eigenfunction-functional-equation} at $z=1$
gives
\begin{equation}
 \sum_{i=0}^nw_iu_i=\varsigma2^{n/2}u_0=0.
 \label{eq:origin-zero-signed-mass}
\end{equation}
Hence the positive and negative weighted masses of $u$ are equal.  If
$u_i\geq0$ for every $i\geq r$, the whole negative mass lies below $r$, so
\begin{equation}
 \sum_{i<r}w_i|u_i|\geq\frac12\sum_{i=0}^nw_i|u_i|.
 \label{eq:last-sign-half-mass}
\end{equation}
For every fixed $c<\delta_\pi$, Theorem~\ref{thm:mass} makes this impossible
when $r\leq\lceil cn\rceil$ and $n$ is sufficiently large.  Letting
$c\uparrow\delta_\pi$ proves the claim for both signs.
\end{proof}

\begin{proposition}[{Krawtchouk sign comparison; cf.~\cite[Chapter~1, Appendix~A]{OpenAI2026}}]
\label{prop:Krawtchouk-sign-comparison}
For every $n\geq1$,
\begin{equation}
 A^{\mathrm K}_+(n)\leq A^{\mathrm K}_-(n).
 \label{eq:Krawtchouk-sign-comparison}
\end{equation}
\end{proposition}

\begin{proof}
Put $z_*=\sqrt2-1$ and associate to a radial vector $u$ the homogeneous
polynomial
\begin{equation}
 \mathcal F_u(X,Y)=\sum_{i=0}^nw_iu_iX^{n-i}Y^i.
 \label{eq:homogeneous-Krawtchouk-polynomial}
\end{equation}
The weighted Krawtchouk self-duality relation and
Equation~\eqref{eq:Krawtchouk-generating-function} give
\begin{equation}
 \mathcal F_{\cK_nu}(X,Y)
 =\mathcal F_u\!\left(\frac{X+Y}{\sqrt2},
                      \frac{X-Y}{\sqrt2}\right).
 \label{eq:homogeneous-Hadamard-action}
\end{equation}

Let $u\in\cE_n^-$ satisfy $u_0=0$ and $r=r_n(u)<\infty$.
Since $\cK_nu=-u$ and the linear involution fixes $(X,z_*X)$,
Equation~\eqref{eq:homogeneous-Hadamard-action} gives
\[
 \mathcal F_u(X,z_*X)
 =-\mathcal F_{\cK_nu}(X,z_*X)
 =-\mathcal F_u(X,z_*X).
\]
Thus $\mathcal F_u$ vanishes on the line $M=0$, where $M=z_*X-Y$,
and therefore has the form $\mathcal F_u=M Q$.  Since $M$ changes sign
under the substitution in Equation~\eqref{eq:homogeneous-Hadamard-action},
$Q$ is invariant.
Since $u_0=0$ and $M(X,0)=z_*X$, one also has $Q(X,0)=0$.  Thus
\begin{equation}
 \widetilde F=-LQ,
\qquad L=X+z_*Y,
 \label{eq:sign-comparison-map}
\end{equation}
is invariant because $L$ is fixed by the same substitution.  It is nonzero
and homogeneous of degree $n$, and its $X^n$ coefficient vanishes.  Let $v$
be the radial vector determined by $\mathcal F_v=\widetilde F$.
Equation~\eqref{eq:homogeneous-Hadamard-action} then gives
$v\in\cE_n^+$, while the preceding observations give $v\ne0$ and $v_0=0$.

It remains to check feasibility for the $+1$-eigenfunction problem defining
$A^{\mathrm K}_+(n)$, namely that $v_i\geq0$ for $i\geq r$.  Write
\[
 \mathcal F_u(X,Y)=X^n\sum_{i=0}^np_it^i,
 \qquad
 Q(X,Y)=X^{n-1}\sum_{i=0}^{n-1}q_it^i,
 \qquad t=Y/X,
\]
so $p_i=w_iu_i$, and set $q_{-1}=q_n=0$.  The factorization by
$M=X(z_*-t)$ gives the identity
\begin{equation}
 p_i=z_*q_i-q_{i-1}.
 \label{eq:sign-comparison-recurrence}
\end{equation}
Since $p_i\geq0$ for $i\geq r$, backward induction from
$p_n=-q_{n-1}$ gives $q_i\leq0$ for every $i\geq r-1$.  The coefficient of
$X^{n-i}Y^i$ in $\widetilde F$ is
\begin{equation}
 w_iv_i=-q_i-z_*q_{i-1}\geq0
 \qquad(i\geq r).
 \label{eq:sign-comparison-tail}
\end{equation}
Thus $r_n(v)\leq r_n(u)$.  Taking the infimum over admissible $u$ proves
Equation~\eqref{eq:Krawtchouk-sign-comparison}.
\end{proof}

\begin{lemma}
\label{lem:finite-sign-alternative}
If $1\leq d<n/2$, exactly one of the following holds:
\begin{enumerate}[label=\textup{(\roman*)}]
 \item there is a nonzero $g\in\cE_n^-$ such that $g_0=0$ and
 $g_i\geq0$ for $d\leq i\leq n$;
 \item there is an $h\in\cE_n^+$ such that
 \begin{equation}
  h_0=1,\qquad h_i=0\quad(1\leq i<d),\qquad
  h_i>0\quad(d\leq i\leq n).
  \label{eq:positive-self-dual-alternative}
 \end{equation}
\end{enumerate}
In case \textup{(ii)}, the vector $A_i=w_ih_i$ is feasible for
$\operatorname{LP}_n(d)$ and has objective $2^{n/2}$.
\end{lemma}

\begin{proof}
Let $V=\{g\in\cE_n^-:g_0=0\}$, and let $W$ be its image under restriction
to the coordinates $d,\ldots,n$.  This restriction is injective.  Indeed,
if such a $g$ were
supported below $d$ and $G_g$ had degree $m<d<n/2$, the functional equation
would force $G_g$ to vanish to order $n-m>m$ at $-1$, and hence $g=0$.

Farkas' lemma, applied with the weighted inner product, says that exactly
one of the following holds: $W$ contains a nonzero nonnegative vector, or
$W^\perp$ contains a strictly positive vector.  By injectivity, the first
case is \textup{(i)}.
In the second, extend the positive vector by zero on $0,\ldots,d-1$.
Because
\[
 V^\perp=\cE_n^++\operatorname{span}\{e_0\},
\]
the extension has the form $h+ce_0$ with $h\in\cE_n^+$.  Thus $h$ vanishes
on layers $1,\ldots,d-1$ and is strictly positive thereafter.  Evaluating
the $+1$ functional equation at $z=1$ gives
\begin{equation}
 \sum_{i=0}^nw_ih_i=2^{n/2}h_0,
 \label{eq:self-dual-objective}
\end{equation}
so tail positivity implies $h_0>0$.  After normalization, the second case
is \textup{(ii)}.

Finally, $A_i=w_ih_i$ has the required nonnegativity and gap.  Using
$w_iK_j^{(n)}(i)=w_jK_i^{(n)}(j)$ and $\cK_nh=h$ gives
\[
 \sum_{i=0}^nA_iK_j^{(n)}(i)
 =2^{n/2}w_jh_j\geq0,
\]
while Equation~\eqref{eq:self-dual-objective} gives its objective.
\end{proof}

\subsection{Why the rate is one half}

Finite linear-programming duality gives the equivalent formulation
\begin{equation}
 \operatorname{LP}_n(d)=
 \min\left\{F(0):
 \begin{array}{l}
 F(i)=\displaystyle\sum_{j=0}^n b_jK_j^{(n)}(i),\quad
 b_0=1,\quad b_j\geq0,\\
 F(i)\leq0\quad(d\leq i\leq n)
 \end{array}\right\}.
 \label{eq:Delsarte-dual}
\end{equation}
Suppose $F$ is feasible for \eqref{eq:Delsarte-dual}.  Put $f_i=F(i)$, and write
$b=(b_0,\ldots,b_n)$.  The expansion of $F$ says
\begin{equation}
 f=\cK_n(2^{n/2}b),
 \qquad \cK_nf=2^{n/2}b.
 \label{eq:dual-Krawtchouk-pair}
\end{equation}
Thus feasibility of $F$ is the pair of sign conditions
$\cK_nf\geq0$ and $f_i\leq0$ for $i\geq d$.  Moreover,
$(\cK_nf)_0=2^{n/2}$ because $b_0=1$.  A dual certificate with objective
below $2^{n/2}$ produces a nonzero $(-1)$-eigenfunction
$g=f-\cK_nf$, which is negative at the origin and nonpositive on layers
$d,\ldots,n$.
Now, if $g$ is a nonzero $(-1)$-eigenfunction with $g(0)\geq0$, then no more than half its mass can occur beyond the last sign change radius.

\begin{proposition}[Delsarte dual reduction]
\label{prop:dual-reduction}
Suppose that every nonzero $g\in\cE_n^-$ satisfies
\begin{equation}
 \sum_{0\leq i<d}w_i|g_i|
 <\frac12\sum_{i=0}^nw_i|g_i|.
 \label{eq:finite-half-mass}
\end{equation}
Then $\operatorname{LP}_n(d)\geq2^{n/2}$.
\end{proposition}

\begin{proof}
Let $F$ be any feasible dual polynomial and use the notation above.  If
$F(0)<2^{n/2}$, set
\begin{equation}
 g=f-\cK_nf.
 \label{eq:dual-minus-eigenfunction}
\end{equation}
Then $\cK_ng=-g$, while dual feasibility gives
$g_i\leq0$ for $i\geq d$.  At the origin,
$g_0=F(0)-2^{n/2}<0$.  Evaluating the functional equation for
$(-1)$-eigenfunctions at $z=1$
gives
\begin{equation}
 \sum_{i=0}^nw_ig_i=-2^{n/2}g_0>0.
 \label{eq:minus-eigenfunction-signed-mass}
\end{equation}
The positive weighted mass of $g$ is therefore larger than its negative
weighted mass.  Every positive entry lies below $d$, so the left side of
\eqref{eq:finite-half-mass} is larger than half the total weighted absolute
mass, a contradiction.
\end{proof}

\begin{lemma}
\label{lem:shortening}
If $A$ is feasible for $\operatorname{LP}_n(d)$ with objective $M$, then
\begin{equation}
 A_i^{\mathrm{sh}}=\frac{n-i}{n}A_i
 \qquad(0\leq i\leq n-1)
 \label{eq:shortening}
\end{equation}
is feasible for $\operatorname{LP}_{n-1}(d)$ and has objective at least
$M/2$.  Consequently, for $0<\delta'<\delta<1/2$,
\begin{equation}
 R_D(\delta)\geq
 1-\frac{\delta}{\delta'}\bigl(1-R_D(\delta')\bigr).
 \label{eq:shortening-rate}
\end{equation}
\end{lemma}

\begin{proof}
This is a special case of the shortening result in \cite[Section~4]{BestBrouwer1977}.  The objective of $A_i^{\mathrm{sh}}$ satisfies
\[
 \sum_iA_i^{\mathrm{sh}}
 =M-\frac1n\sum_i iA_i
 =\frac M2+\frac1{2n}\sum_iA_iK_1^{(n)}(i)
 \geq\frac M2.
\]
Choose a sequence $n_k\to\infty$ realizing the limsup in
$R_D(\delta')$, set $d_k=\lceil\delta'n_k\rceil$, and put
$m_k=\lfloor d_k/\delta\rfloor$.  Then
$m_k/n_k\to\delta'/\delta$ and $\lceil\delta m_k\rceil\leq d_k$.
After $n_k-m_k$ shortenings,
\[
 \operatorname{LP}_{m_k}(\lceil\delta m_k\rceil)
 \geq 2^{-(n_k-m_k)}\operatorname{LP}_{n_k}(d_k).
\]
Divide by $m_k$ and take the limsup to obtain
Equation~\eqref{eq:shortening-rate}.
\end{proof}

\begin{corollary}[rate-one-half lower bound]
\label{cor:lower}
\[ R_D(\delta_\pi)\geq1/2. \]
\end{corollary}

\begin{proof}
Fix $c<\delta_\pi$ and set $d=\lceil cn\rceil$.  For large $n$,
Theorem~\ref{thm:mass} gives \eqref{eq:finite-half-mass} for every nonzero
$g\in\cE_n^-$.  Proposition~\ref{prop:dual-reduction} therefore gives
$\operatorname{LP}_n(d)\geq2^{n/2}$, hence $R_D(c)\geq1/2$.
Apply Equation~\eqref{eq:shortening-rate} with $\delta'=c$ and let
$c\uparrow\delta_\pi$.
\end{proof}

\subsection{The upper bound on \texorpdfstring{$J$}{J}}
\label{sec:matrix-inequality}

Recall from Equation~\eqref{eq:local-channel-coordinates} that, for a
normalized Horn pair,
\begin{equation}
 \Gamma(K,L)=2\tr(K^{1/2}L^{1/2}),\qquad
 \Phi(K,L)=\Ent(K)+\Ent(L)-\Ent(K+L).
 \label{eq:recalled-channel-coordinates}
\end{equation}
Put $G=K+L$ and $Q=K-L$.  The functional to be bounded is
\begin{equation}
 J(K,L)=\tr(K^{1/2}L^{1/2})
 +\frac12\tr\!\left(G^{-1/2}QG^{-1/2}Q\right),
 \label{eq:J-body}
\end{equation}
where the inverse is taken on $\operatorname{supp}G$.

\begin{lemma}[dilution]
\label{lem:Horn-dilution}
If $(K,L)$ is a normalized Horn pair, $0<\delta<1/2$, and
$1-2\delta<\Gamma(K,L)$, then
\begin{equation}
 R_D(\delta)\leq
 \frac{1-2\delta}{\Gamma(K,L)}\Phi(K,L).
 \label{eq:Horn-dilution}
\end{equation}
\end{lemma}

\begin{proof}
Write $\Gamma=\Gamma(K,L)$ and $\Phi=\Phi(K,L)$.  For
$(1-2\delta)/\Gamma<s<1$, the pair
\[
 K_s=sK\oplus[1-s],\qquad L_s=sL\oplus[0]
\]
is normalized and satisfies
$\Gamma(K_s,L_s)=s\Gamma$ and $\Phi(K_s,L_s)=s\Phi$.
Theorem~\ref{thm:realization} therefore gives
$R_D(\delta)\leq s\Phi$.  Letting
$s\downarrow(1-2\delta)/\Gamma$ proves the claim.
\end{proof}

\begin{corollary}
\label{cor:half-rate-matrix-inequality}
For every normalized Horn pair,
\begin{equation}
 \Gamma(K,L)>\frac2\pi
 \quad\Longrightarrow\quad
 \Phi(K,L)>\frac12.
 \label{eq:half-rate-matrix-inequality}
\end{equation}
\end{corollary}

\begin{proof}
Lemma~\ref{lem:Horn-dilution} and Corollary~\ref{cor:lower} give
\[
 \frac12\leq R_D(\delta_\pi)
 \leq\frac{2}{\pi\Gamma(K,L)}\Phi(K,L).
\]
Thus $\Phi(K,L)\geq\pi\Gamma(K,L)/4>1/2$.
\end{proof}

For a pair $(K,L)$ define two Gram matrices
\begin{equation}
 \mathcal G_\pm=
 \begin{pmatrix}
 K&\mathord\pm K^{1/2}L^{1/2}\\
 \mathord\pm L^{1/2}K^{1/2}&L
 \end{pmatrix}
 \label{eq:Gram-pm}
\end{equation}
and the matrices
$K^\perp=\mathcal G_+/2$, $L^\perp=\mathcal G_-/2$.

\begin{proposition}[Gram symmetrization]
\label{prop:duality}
The pair $(K^\perp, L^\perp)$ is normalized and satisfies
\begin{align}
 \Phi(K^\perp,L^\perp)&=1-\Phi(K,L),
 \label{eq:dual-entropy}\\
 \Gamma(K^\perp,L^\perp)
 &=\tr(G^{-1/2}QG^{-1/2}Q).
 \label{eq:dual-Gamma}
\end{align}
Consequently the Gram-symmetrized pair
\begin{equation}
 \widehat K=\frac12K\oplus\frac12K^\perp,
 \qquad
 \widehat L=\frac12L\oplus\frac12L^\perp
 \label{eq:canonical-direct-sum}
\end{equation}
has
\begin{equation}
 \Phi(\widehat K,\widehat L)=\frac12,
 \qquad
 \Gamma(\widehat K,\widehat L)=J(K,L).
 \label{eq:canonical-direct-sum-coordinates}
\end{equation}
\end{proposition}

\begin{proof}
Write $Zx=(K^{1/2}x,L^{1/2}x)$ and
$Z_-x=(K^{1/2}x,-L^{1/2}x)$.  Then
$\mathcal G_+=ZZ^*$, $\mathcal G_-=Z_-Z_-^*$, while
$Z^*Z=Z_-^*Z_-=G$.  Thus both Gram matrices have the same nonzero spectrum
as $G$.  Since
$(\mathcal G_++\mathcal G_-)/2=\operatorname{diag}(K,L)$, entropy scaling
gives Equation~\eqref{eq:dual-entropy}.

On their supports,
$\mathcal G_+^{1/2}=ZG^{-1/2}Z^*$ and
$\mathcal G_-^{1/2}=Z_-G^{-1/2}Z_-^*$.  Cyclicity of trace and
$Z^*Z_-=Q$ prove Equation~\eqref{eq:dual-Gamma}.  Since $\Phi$ and
$\Gamma$ are additive under orthogonal direct sums and homogeneous,
Equation~\eqref{eq:canonical-direct-sum-coordinates} follows.
\end{proof}

\begin{corollary}
\label{cor:J-upper}
For every normalized Horn pair $(K,L)$,
\begin{equation}
 J(K,L)\leq\frac2\pi.
 \label{eq:J-upper}
\end{equation}
\end{corollary}

Indeed, if $J(K,L)>2/\pi$, then
Equation~\eqref{eq:canonical-direct-sum-coordinates} gives a normalized
positive pair $(\widehat K,\widehat L)$ with $\Phi=1/2$ and $\Gamma>2/\pi$,
contradicting Corollary~\ref{cor:half-rate-matrix-inequality}.

\section{Construction of Horn pairs approaching \texorpdfstring{$J=2/\pi$}{J = 2/pi}}
\label{sec:recovery}

We now construct finite-dimensional normalized Horn pairs $(K_N,L_N)$ for
which $J(K_N,L_N)\to2/\pi$.  Together with
Corollary~\ref{cor:J-upper}, this determines the supremum of
$J$ over vector spaces of arbitrary dimension.  Section~\ref{sec:closing} then converts the same sequence into the
upper bound for $R_D$.

Let
\begin{equation}
 J_{\rm fin}=
 \sup\left\{J(K,L):
 \begin{array}{c}
 K,L\succeq0\text{ on a finite-dimensional Hilbert space},\\
 \tr(K+L)=1
 \end{array}\right\}.
 \label{eq:J-all-dimensional}
\end{equation}
Corollary~\ref{cor:J-upper} gives $J_{\rm fin}\leq2/\pi$, and we must
construct a sequence proving equality.  The feasible set in each fixed
dimension is convex, being a semidefinite slice, but $J$ is not concave.
Moreover, the following proposition shows that our construction must not
allow $K$ and $L$ to be simultaneously diagonalizable: the relative position
of $K$ and $L$ is important as well.

\begin{proposition}[Commuting pairs do not suffice]
\label{prop:commuting-ceiling}
If $K$ and $L$ commute, then
\begin{equation}
 J(K,L)\leq\frac58<\frac2\pi.
 \label{eq:commuting-ceiling}
\end{equation}
More precisely, after simultaneous diagonalization, put
\begin{equation}
 s_i=k_i+l_i,\qquad
 p_i=\frac{\sqrt{k_il_i}}{s_i}
 \label{eq:commuting-coordinates}
\end{equation}
on the indices for which $s_i>0$.  Then
\begin{equation}
 \frac58-J(K,L)
 =2\sum_i s_i\left(p_i-\frac14\right)^2.
 \label{eq:commuting-stability}
\end{equation}
\end{proposition}

\begin{proof}
Since $\sum_i s_i=1$, direct substitution in Equation~\eqref{eq:J} gives
\[
 J(K,L)=\sum_i s_i\left(\frac12+p_i-2p_i^2\right).
\]
The identity
\[
 \frac58-\left(\frac12+p-2p^2\right)
 =2\left(p-\frac14\right)^2
\]
proves Equation~\eqref{eq:commuting-stability} and the bound.
\end{proof}

Passing to twice the dimension, every positive pair has the following projective normal form.

\begin{proposition}[projective normal form]
\label{prop:projective-normal-form}
Let $K,L\succeq0$ on $\mathcal H$ and $\tr(K+L)=1$.  There are a Hilbert
space $\widetilde{\mathcal H}$ of dimension at most
$2\dim\mathcal H$, an isometry $V:\mathcal H\to\widetilde{\mathcal H}$,
a state $G$ on $\widetilde{\mathcal H}$, and an orthogonal projection $P$
such that
\begin{equation}
 VKV^*=G^{1/2}PG^{1/2},\qquad
 VLV^*=G^{1/2}(I-P)G^{1/2}.
 \label{eq:projective-normal-form}
\end{equation}
Conversely, every state $G$ and projection $P$ give a normalized positive
pair by the right side of Equation~\eqref{eq:projective-normal-form}.
\end{proposition}

\begin{proof}
Put $G_0=K+L$ and let
\[
 D=G_0^{-1/2}KG_0^{-1/2}.
\]
The inverse is taken on $\operatorname{supp}G_0$, and $D$ is extended by
zero on $\ker G_0$.  Then $0\preceq D\preceq I$.  On
$\widetilde{\mathcal H}=\mathcal H\oplus\mathcal H$, define
\[
 Vx=\binom{D^{1/2}x}{(I-D)^{1/2}x},\qquad
 P=\begin{pmatrix}I&0\\0&0\end{pmatrix},\qquad
 G=VG_0V^*.
\]
The map $V$ is an isometry and $V^*PV=D$.  Functional calculus on
$\operatorname{ran}V$ gives $G^{1/2}=VG_0^{1/2}V^*$, while
$\tr G=\tr G_0=1$.  Hence
\[
 G^{1/2}PG^{1/2}=VKV^*,\qquad
 G^{1/2}(I-P)G^{1/2}=VLV^*.
\]
The converse is immediate: the two matrices are positive, have sum $G$,
and therefore have trace sum one.
\end{proof}

We note that it is a straightforward observation that $J$ is unchanged under embedding by an isometry $V$, so optimization over pairs $(K, L)$ can be reduced to an optimization over pairs $(G, P)$.  Here and below, a state means a positive semidefinite operator of trace one.  For a state $G$ and an orthogonal projection $P$, write
\[
 K_{G,P}=G^{1/2}PG^{1/2},\qquad
 L_{G,P}=G^{1/2}(I-P)G^{1/2}.
\]
The following ``reflection pair'' identity allows us to directly write the objective $J$ in terms of $G$ and $P$, eliminating $K$ and $L$.  In the following, all operators live on an ambient finite-dimensional space $\mathcal{H}$, and all inverse square roots below are taken on the corresponding supports.

\begin{proposition}[reflection pair identity]
\label{prop:reflection-pair}
Let $G$ be a state, let $R=R^*=R^{-1}$ be a reflection, and put
\begin{equation}
 P_\pm=\frac{I\pm R}{2},\qquad
 G_e=\frac{G+RGR}{2},\qquad
 G_o=\frac{G-RGR}{2}.
 \label{eq:reflection-even-odd}
\end{equation}
For $K=K_{G,P_+}$ and $L=L_{G,P_+}$, one has
\begin{align}
 2\tr(K^{1/2}L^{1/2})
 &=\tr\!\left(G_e^{-1/2}G_oG_e^{-1/2}G_o\right),
 \label{eq:reflection-Gamma}\\
 \tr\!\left(G^{-1/2}(K-L)G^{-1/2}(K-L)\right)
 &=\tr\!\left(RG^{1/2}RG^{1/2}\right).
 \label{eq:reflection-bias}
\end{align}
Consequently,
\begin{equation}
 J(K,L)=J\!\left(\frac G2,\frac{RGR}{2}\right),
 \label{eq:reflection-pair-identity}
\end{equation}
with the two terms of $J$ exchanged.
\end{proposition}

\begin{proof}
Relative to the decomposition
$\mathcal H=P_+\mathcal H\oplus P_-\mathcal H$, write
\[
 A=P_+GP_+,\qquad B=P_-GP_-,\qquad C=P_+GP_-.
\]
Then
\[
 G_e=A\oplus B,\qquad
 G_o=\begin{pmatrix}0&C\\ C^*&0\end{pmatrix},
\]
and hence
\begin{equation}
 \tr\!\left(G_e^{-1/2}G_oG_e^{-1/2}G_o\right)
 =2\tr\!\left(A^{-1/2}CB^{-1/2}C^*\right).
 \label{eq:reflection-block-bias}
\end{equation}
On the other hand, the polar decomposition of $G^{1/2}P_+$ gives
\[
 K^{1/2}=G^{1/2}P_+A^{-1/2}P_+G^{1/2},
\]
and the analogous identity holds for $L$.  Multiplication and cyclicity of
trace show that the right side of
Equation~\eqref{eq:reflection-block-bias} is
$2\tr(K^{1/2}L^{1/2})$, proving
Equation~\eqref{eq:reflection-Gamma}.

For the second identity, set $S=G^{1/2}$ and let $\Pi$ be the support
projection of $G$.  Since $K-L=SRS$ and
$G^{-1/2}S=SG^{-1/2}=\Pi$, cyclicity and $S\Pi=S$ give
\[
 \tr\!\left(G^{-1/2}(K-L)G^{-1/2}(K-L)\right)
 =\tr(RSRS),
\]
which is Equation~\eqref{eq:reflection-bias}.  Finally,
$(RGR)^{1/2}=RSR$, while the sum and difference of $G/2$ and $RGR/2$
are $G_e$ and $G_o$, respectively.  Equations
\eqref{eq:reflection-Gamma} and \eqref{eq:reflection-bias} therefore give
Equation~\eqref{eq:reflection-pair-identity}.
\end{proof}

Combining Propositions~\ref{prop:projective-normal-form} and
\ref{prop:reflection-pair} gives the equivalent formulation
\begin{equation}
 J_{\rm fin}
 =\sup_{\substack{\mathcal H\ \text{ finite-dimensional}\\
                   G\succeq0,\ \tr G=1\\R=R^*=R^{-1}}}
 J\!\left(\frac G2,\frac{RGR}{2}\right).
 \label{eq:J-reflection-optimization}
\end{equation}
The projective form has the channel interpretation discussed in
Subsection~\ref{subsec:projective-PGM}, while the reflection form will be
used to evaluate the construction below.

We will construct a state $G_N\in M_{2^N}$ and a projection $P_N$, with
$K_N=K_{G_N,P_N}$ and $L_N=L_{G_N,P_N}$.  Theorem~\ref{thm:explicit-recovery}
will prove
\begin{equation}
 J_N:=J(K_N,L_N)
 =\frac1{2N}\cot\frac{\pi}{4N}
 =\frac2\pi-\frac{\pi}{24N^2}+O(N^{-4}).
 \label{eq:recovery-goal}
\end{equation}
Proposition~\ref{prop:duality} gives a Gram-symmetrized pair whose
$(\Phi,\Gamma)$ parameters are $(1/2,J_N)$.  For
$\delta>\delta_\pi$, Section~\ref{sec:closing} finally applies
Theorem~\ref{thm:realization} to obtain the quantitative
upper bound in Theorem~\ref{thm:upper}.

The constructions in this section run parallel to OpenAI's Chapter~2,
especially Sections~8.2--8.3 which address arbitrarily deep levels of the
moving subspace hierarchy for spherical codes \cite{OpenAI2026}.  At level $N$
their construction uses interlacing poles $x_1,\ldots,x_N$ and zeros
$y_1,\ldots,y_{N-1}$ of a rational Stieltjes transform.  A relative entropy
argument bounds their optimization problem above by a quantity $\lambda_*$.
To recover $\lambda_*$ in the limit, they specialize their construction to
the roots and critical points of Chebyshev polynomials.

\subsection{Projective channels and the pretty good measurement}
\label{subsec:projective-PGM}

For two equiprobable channel outputs $\sigma_0,\sigma_1$, put
\begin{equation}
 K=\frac{\sigma_0}{2},\qquad L=\frac{\sigma_1}{2},\qquad
 G=K+L=\frac{\sigma_0+\sigma_1}{2}.
 \label{eq:channel-positive-pair}
\end{equation}
The effects of their pretty good measurement are
\begin{equation}
 M_0=G^{-1/2}KG^{-1/2},\qquad
 M_1=G^{-1/2}LG^{-1/2},
 \label{eq:PGM-effects}
\end{equation}
with inverses taken on $\operatorname{supp}G$.  We call a binary-input
classical--quantum channel \emph{projective} if the two effects of its pretty
good measurement are complementary orthogonal projections.  In the coordinates
$K=K_{G,P}$ and $L=L_{G,P}$, the pretty good measurement is the
compression of $(P,I-P)$ to $\operatorname{supp}G$, and it is exactly
$(P,I-P)$ when $G$ is faithful.
Moreover,
\begin{align}
 \Phi(K,L)&=1-\chi(\sigma_0,\sigma_1),
 \label{eq:Phi-Holevo}\\
 \tr\!\left(G^{-1/2}(K-L)G^{-1/2}(K-L)\right)
 &=1-2p_{\rm pgm}(\sigma_0,\sigma_1).
 \label{eq:PGM-bias}
\end{align}
Here
\begin{equation}
 \chi(\sigma_0,\sigma_1)
 =\Ent(G)-\frac12\Ent(\sigma_0)-\frac12\Ent(\sigma_1),
 \qquad
 p_{\rm pgm}=\tr(KM_1)+\tr(LM_0).
 \label{eq:channel-quantities}
\end{equation}
The first identity follows from entropy scaling; the second is the difference
between the PGM success and error probabilities.  Proposition~\ref{prop:duality}
therefore gives
\begin{equation}
 \Phi(K^\perp,L^\perp)=\chi(\sigma_0,\sigma_1),\qquad
 \Gamma(K^\perp,L^\perp)=1-2p_{\rm pgm}(\sigma_0,\sigma_1).
 \label{eq:PGM-pair-dictionary}
\end{equation}
Applying Theorem~\ref{thm:realization} to the pair $(K^\perp, L^\perp)$ shows
that $p_{\rm pgm}(\sigma_0,\sigma_1)<\delta$ implies
$R_D(\delta)\leq\chi(\sigma_0,\sigma_1)$.
This proves the asymptotic form of
the pretty good criterion~\cite[Theorem~1]{AlrabiahGuruswami2026},
and moreover shows that the pretty good criterion bounds the Delsarte bound $R_D$.

We now express the binary pure-state channel of Alrabiah and Guruswami in
these projective coordinates.  Let
$\mathsf X,\mathsf Y,\mathsf Z$ be the Pauli matrices and, for
$0\leq r<1/2$, put
\begin{equation}
 c=1-2r,\qquad s=2\sqrt{r(1-r)},\qquad
 G_r=\frac{I+s\mathsf X}{2},\qquad
 R=\mathsf Z,\qquad P_\pm=\frac{I\pm R}{2}.
 \label{eq:PSC-projective-data}
\end{equation}
Since $c^2+s^2=1$, direct calculation gives
\begin{align}
 K_r:=K_{G_r,P_+}&=\frac14(I+s\mathsf X+c\mathsf Z)
                  =\frac12|\psi_0\rangle\langle\psi_0|,\notag\\
 L_r:=L_{G_r,P_+}&=\frac14(I+s\mathsf X-c\mathsf Z)
                  =\frac12|\psi_1\rangle\langle\psi_1|,
 \label{eq:PSC-positive-pair}
\end{align}
where
\begin{equation}
 |\psi_0\rangle=\sqrt{1-r}\,|0\rangle+\sqrt r\,|1\rangle,
 \qquad |\psi_1\rangle=\mathsf X|\psi_0\rangle.
 \label{eq:PSC-states}
\end{equation}
Thus $2K_r,2L_r$ are the outputs of the pure-state channel, and their PGM is
the computational-basis measurement.  In particular,
\begin{equation}
 p_{\rm pgm}=r,\qquad
 \chi=h_2\!\left(\frac{1+s}{2}\right),\qquad
 \Gamma(K_r,L_r)=s^2,\qquad
 J(K_r,L_r)=\frac{s^2+c}{2}.
 \label{eq:PSC-statistics}
\end{equation}

\subsection{The \texorpdfstring{$N$}{N}-qubit projection}
\label{subsec:N-qubit}

The $N$-qubit construction that we consider is a generalization of the
pure-state channel and recovers it in the special case $N=1$.
Let $\eta\in(-1,1)^N$ and let $\gamma\in\RR^N$ be a unit vector.  We use
these data to choose the state and projection in
Proposition~\ref{prop:projective-normal-form}.  The next subsection will select $(\eta,\gamma)$.

Recall that for $V$ a $2N$-dimensional vector space with nondegenerate
symmetric bilinear form $\langle\cdot,\cdot\rangle$, the Clifford algebra is
\[
 \operatorname{Cl}(V)
 =T(V)\big/\bigl\langle
 xy+yx-2\langle x,y\rangle1:x,y\in V
 \bigr\rangle.
\]
For elements $A$ and $B$ of an algebra, write $\{A,B\}=AB+BA$ for their
anticommutator.  An orthonormal basis
$a_1,\ldots,a_N,b_1,\ldots,b_N$ of $V$ satisfies
\begin{equation}
 \{a_j,a_k\}=2\delta_{jk},
 \qquad
 \{b_j,b_k\}=2\delta_{jk},
 \qquad
 \{a_j,b_k\}=0.
 \label{eq:Majorana-relations}
\end{equation}

This basis identifies $V=A\oplus B$ with two orthogonal copies of
$\RR^N$.  Write
\[
 a(w)=\sum_{j=1}^Nw_ja_j,\qquad
 b(w)=\sum_{j=1}^Nw_jb_j
\]
for the corresponding linear isometric embeddings into
$V\subseteq\operatorname{Cl}(V)$.  Their images
generate two Clifford subalgebras of $\operatorname{Cl}(V)$.

On $\mathcal H_N=(\CC^2)^{\otimes N}$, we have a Jordan--Wigner realization
\begin{equation}
 a_j=\mathsf X_1\cdots\mathsf X_{j-1}\mathsf Z_j,
 \qquad
 b_j=\mathsf X_1\cdots\mathsf X_{j-1}\mathsf Y_j,
 \qquad
 \iu a_jb_j=\mathsf X_j.
 \label{eq:Clifford-generators}
\end{equation}
For $\eta=(\eta_1,\ldots,\eta_N)\in(-1,1)^N$, consider
\begin{equation}
 G_\eta=2^{-N}\prod_{j=1}^N(I-\eta_j\mathsf X_j).
 \label{eq:product-state}
\end{equation}
Put $r_j=\operatorname{arctanh}\eta_j$.  Functional calculus for $\mathsf X_j$ gives
\[
 \frac{I-\eta_j\mathsf X_j}{2}
 =\frac{e^{-r_j\mathsf X_j}}{2\cosh r_j},
\]
and hence
\begin{equation}
 G_\eta
 =\frac{e^{-H_\eta}}{\tr(e^{-H_\eta})},
 \qquad
 H_\eta=\sum_{j=1}^Nr_j\mathsf X_j.
 \label{eq:quadratic-Gibbs}
\end{equation}
The factors of $G_\eta$ are commuting positive definite matrices and
$\tr G_\eta=1$.  At the same time, in the joint eigenbasis of $\mathsf X_j$, the state $G_{\eta}$ is diagonal and a classical probability on spins $x \in \{ \pm 1 \}^N$, and moreover $\mathbb{P}(x_j = \pm 1) = (1\mp\eta_j) / 2$.  The first two moments of this probability measure are straightforward to compute:
\begin{equation}
  \mathbb{E}[\mathsf X_i] = -\eta_i, \qquad \mathrm{Cov}(\mathsf X) = (I-D_{\eta}^2) \quad\text{where}\quad D_\eta=\operatorname{diag}(\eta_1,\ldots,\eta_N).
  \label{eq:first-second-moments}
\end{equation}

We define the reflection and projection
\begin{equation}
 R_\gamma=a(\gamma),
 \quad
 P_\gamma^\pm=\frac{I\pm R_\gamma}{2},
 \label{eq:product-density}
\end{equation}
where the Clifford relations give $R_\gamma^2=I$.  Hence
$(G_\eta,R_\gamma)$ is feasible for
Equation~\eqref{eq:J-reflection-optimization}.

\begin{lemma}
\label{lem:reflection-action}
The embedding $a$ intertwines the Householder reflection
$I-2\gamma\gamma^{\mathsf T}$ with the Pin action
$X\mapsto-R_\gamma X R_\gamma$, while the Pin action fixes the image of $b$.
That is, for every $w\in\RR^N$,
\begin{equation}
 -R_\gamma a(w)R_\gamma
 =a\!\left((I-2\gamma\gamma^{\mathsf T})w\right),\qquad
 -R_\gamma b(w)R_\gamma=b(w).
 \label{eq:reflection-action}
\end{equation}
\end{lemma}

\begin{proof}
The Clifford relations give
$\{R_\gamma,a(w)\}=2\langle\gamma,w\rangle$ and
$\{R_\gamma,b(w)\}=0$.  Multiplying these identities by $R_\gamma$ and
using $R_\gamma^2=I$ proves Equation~\eqref{eq:reflection-action}.
\end{proof}

Writing
$G_\eta=2^{-N}\prod_j(I-\iu\eta_ja_jb_j)$, the lemma shows that
ordinary conjugation contributes one minus sign on each of the $a$- and
$b$-generators.  These signs cancel in their quadratic products, so it
replaces the coupling matrix $D_\eta$ by
$(I-2\gamma\gamma^{\mathsf T})D_\eta$.
Put
\begin{equation}
 \xi_{\eta,\gamma}=D_\eta\gamma,
 \qquad
 T_{\eta,\gamma}=(I-\gamma\gamma^{\mathsf T})D_\eta.
 \label{eq:small-data}
\end{equation}
Since $D_\eta$ is self-adjoint,
\begin{equation}
 D_\eta=\gamma\xi_{\eta,\gamma}^{\mathsf T}+T_{\eta,\gamma}.
 \label{eq:one-particle-splitting}
\end{equation}
Applying reflection across $\gamma^{\perp}$,
\begin{equation}
 (I-2\gamma\gamma^{\mathsf T})D_\eta
 =T_{\eta,\gamma}-\gamma\xi_{\eta,\gamma}^{\mathsf T}.
 \label{eq:reflected-coupling}
\end{equation}
Define the defect matrix
\begin{equation}
 \Lambda_{\eta,\gamma}
 =I-T_{\eta,\gamma}^{\mathsf T}T_{\eta,\gamma}
 =I-D_\eta^2
  +\xi_{\eta,\gamma}\xi_{\eta,\gamma}^{\mathsf T}.
 \label{eq:defect-matrix}
\end{equation}
The following computation writes $J$ in terms of two expressions $v$ and $u$.  The term $v$ is a quadratic form of the ordinary covariance of~\eqref{eq:first-second-moments} with exponent $+1/2$, while the term $u$ is a quadratic form of the spiked, i.e., rank-one perturbed, covariance given by the defect matrix~\eqref{eq:defect-matrix} with exponent $-1/2$.

\begin{theorem}
\label{thm:compression}
For the state $G_\eta$ in Equation~\eqref{eq:product-state} and the
projection $P_\gamma^+$ in Equation~\eqref{eq:product-density}, write
\[
 K=K_{G_\eta,P_\gamma^+},
 \qquad
 L=L_{G_\eta,P_\gamma^+},
\]
and let
\begin{align}
 u(\eta,\gamma)&=2\tr(K^{1/2}L^{1/2}),\notag\\
 v(\eta,\gamma)&=\tr\!\left(
 R_\gamma G_\eta^{1/2}R_\gamma G_\eta^{1/2}\right).
\end{align}
Then
\begin{align}
 u(\eta,\gamma)
 &=\xi_{\eta,\gamma}^{\mathsf T}
   \Lambda_{\eta,\gamma}^{-1/2}\xi_{\eta,\gamma},
 \label{eq:u-compression}\\
 v(\eta,\gamma)
 &=\gamma^{\mathsf T}(I-D_\eta^2)^{1/2}\gamma.
 \label{eq:v-compression}
\end{align}
Thus $J(K,L)=(u+v)/2$.
\end{theorem}

\begin{proof}
To simplify notation, we will write $G=G_\eta$, $R=R_\gamma$,
$P_\pm=P_\gamma^\pm$, $D=D_\eta$, $T=T_{\eta,\gamma}$,
$\xi=\xi_{\eta,\gamma}$, and
$\Lambda=\Lambda_{\eta,\gamma}$.
By convention, inverse square roots below are taken on the corresponding
supports.

We first recall some basic properties of Clifford algebras $\operatorname{Cl}(V)$.  The symbol map $\sigma \colon \bigwedge V \rightarrow \operatorname{Cl}(V)$ sends $v_1 \wedge \dots \wedge v_n$ to the product $v_1 \cdots v_n$ when $v_i$ are mutually orthogonal.  The normalized trace $\tau=2^{-N}\tr$ has the property that $\tau(\sigma(v))$ is the degree-zero part of $v$.  In $\bigwedge V$, we let $\exp_{\wedge}(\theta) = \sum_{n=0}^{\infty} \frac{\theta^{\wedge n}}{n!}$.  Finally, contraction $\iota_v$ is the graded derivation on $\bigwedge V$ defined by $\iota_v(w)=\langle v,w\rangle$ for vectors $v,w$.

Equation~\eqref{eq:reflection-Gamma} gives
\begin{equation}
 u=\tr\!\left(G_e^{-1/2}G_oG_e^{-1/2}G_o\right),
 \label{eq:u-even-odd}
\end{equation}
where $G_e$ and $G_o$ are the even and odd parts of $G$ under conjugation
by $R$, as in Equation~\eqref{eq:reflection-even-odd}.  Thus, we aim to find a simple expression for $G_e$ and $G_o$.  For an $N \times N$ matrix $M$, define
\[
  \omega_M = \sum_{i,j} M_{ij} a_i \wedge b_j \in \bigwedge V.
\]
Then, by linearity
\begin{equation}
  \begin{aligned}
  \omega_D &= \omega_T + \omega_{\gamma \xi} = \omega_T + R \wedge b(\xi), \\
  \exp_{\wedge}(-\iu \omega_D) &= \exp_{\wedge}(-\iu \omega_T) - \iu R \wedge b(\xi) \wedge \exp_{\wedge}(-\iu \omega_T).
  \end{aligned}
  \label{eq:omega-exponential}
\end{equation}

Now, recall that
\begin{equation}
  G = 2^{-N} \prod_j(I-\iu\eta_ja_jb_j) = 2^{-N} \sigma\left(\exp_{\wedge}(-\iu \omega_D) \right),
  \label{eq:wick-for-g}
\end{equation}
which is $2^{-N}$ times the symbol of the left hand side of \eqref{eq:omega-exponential}.  We claim the two summands identify with $G_e$ and $G_o$, respectively.  Indeed, the image of $T$ lies in $\gamma^\perp$, so every vector occurring in $\omega_T$ is orthogonal to $R$.  Lemma~\ref{lem:reflection-action} shows that conjugation by $R$ negates each such vector.  It therefore fixes the symbol of $\exp_{\wedge}(-\iu\omega_T)$, whose terms have even degree.  In the second summand, $R$ itself is fixed, while $b(\xi)$ and every vector in $\omega_T$ are negated; each term has an odd number of these latter vectors.  Thus conjugation by $R$ negates its symbol.

For every vector $v$ and even
exterior form $z$, the Clifford relations give
$\sigma(v\wedge z)=\tfrac12\{v,\sigma(z)\}$.  Also, $R$ is orthogonal to every
vector occurring in $b(\xi)\wedge \exp_{\wedge}(-\iu\omega_T)$ due to the projection $T$, so its wedge product becomes
ordinary left multiplication after applying $\sigma$.  Consequently,
\[
 \sigma\bigl(R\wedge b(\xi)\wedge \exp_{\wedge}(-\iu\omega_T)\bigr)
 =R\,\sigma\bigl(b(\xi)\wedge \exp_{\wedge}(-\iu\omega_T)\bigr)
 =\frac{R}{2}\{b(\xi),\sigma(\exp_{\wedge}(-\iu\omega_T))\}.
\]
The odd summand of \eqref{eq:omega-exponential}, multiplied by
$2^{-N}$, is therefore
$-\iu R\,\tfrac12\{b(\xi),G_e\}$, since
$G_e=2^{-N}\sigma(\exp_{\wedge}(-\iu\omega_T))$.  Thus, to summarize, we have
\begin{equation}
 G_e=2^{-N}\sigma\left(\exp_{\wedge}(-\iu \omega_T)\right),
 \qquad
 G_o=-\iu\,R\cdot\tfrac12\{b(\xi),G_e\}.
 \label{eq:arrow-even-odd}
\end{equation}

For a real two-form $\omega\in\bigwedge^2V$, let
\[
 K_\omega w=\iota_w\omega,\qquad
 X_\omega=\sigma\!\left(\exp_{\wedge}(-\iu\omega)\right).
\]
Here $\iota_w$ denotes contraction by $w$.

\begin{lemma}
\label{lem:gaussian-sandwich}
If $\|K_\omega\|<1$, then $X=X_\omega$ is positive definite,
$\tau(X)=1$, and, with $S=(I+K_\omega^2)^{1/2}$,
\begin{equation}
 X^{1/2}wX^{1/2}=\tfrac12\{Sw,X\},\qquad
 X^{-1/2}\tfrac12\{w,X\}X^{-1/2}=S^{-1}w
 \quad(w\in V).
 \label{eq:gaussian-sandwich}
\end{equation}
Moreover, for $v,w\in V$,
\begin{equation}
 \tau\!\left(\tfrac12\{v,X\}w\right)=\langle v,w\rangle.
 \label{eq:gaussian-trace-pairing}
\end{equation}
\end{lemma}

\begin{proof}
Choose an orthogonal decomposition of $V$ into invariant planes for
the skew operator $K_\omega$.  On one such plane, with orthonormal
basis $e,f$, write $\omega=s\,e\wedge f$ and
$\kappa=\iu ef$.  Then $X=I-s\kappa$ on this plane,
$S=\sqrt{1-s^2}\,I$, and $\kappa$ anticommutes with every
vector $w$ in the plane.  Thus
\[
 (I-s\kappa)^{1/2}w(I-s\kappa)^{1/2}
 =\sqrt{1-s^2}\,w
 =\tfrac12\{\sqrt{1-s^2}\,w,I-s\kappa\}.
\]
For orthogonal planes the even factors $I-s\kappa$ commute.
A vector in one plane commutes with all the other factors, so the
one-plane identity gives $X^{1/2}wX^{1/2}=\tfrac12\{Sw,X\}$ on every
plane, and hence on $V$ by linearity.  Each factor is positive because
$|s|<1$, and the degree-zero part of the exterior exponential is $1$,
so $\tau(X)=1$.  Replacing $w$ in the first identity by $S^{-1}w$
and multiplying by $X^{-1/2}$ gives the second identity.
Finally, cyclicity of trace and $\{v,w\}=2\langle v,w\rangle I$
give
\[
 \tau\!\left(\tfrac12\{v,X\}w\right)
 =\tfrac12\tau\!\left(X(wv+vw)\right)
 =\langle v,w\rangle\tau(X)
 =\langle v,w\rangle.
\]
\end{proof}

For the form $\omega_M=\sum_{i,j}M_{ij}a_i\wedge b_j$, contraction gives
\[
 K_{\omega_M}a(x)=b(M^{\mathsf T}x),\qquad
 K_{\omega_M}b(y)=-a(My).
\]
Consequently $\|K_{\omega_M}\|=\|M\|$, and
\begin{equation}
 (I+K_{\omega_M}^2)a(x)=a((I-MM^{\mathsf T})x),
 \qquad
 (I+K_{\omega_M}^2)b(y)=b((I-M^{\mathsf T}M)y).
 \label{eq:gaussian-defect-action}
\end{equation}

To compute $u$, set $X=X_{\omega_T}=2^NG_e$ and
$Y=\tfrac12\{b(\xi),X\}$.  Equation~\eqref{eq:arrow-even-odd} says
$G_o=-\iu\,2^{-N}RY$.  Because $\omega_T$ has no $R$ component,
$R$ commutes with $X$ and anticommutes with $Y$.  Thus
Equation~\eqref{eq:u-even-odd} becomes
\[
 u=\tau\!\left(X^{-1/2}YX^{-1/2}Y\right).
\]
Now $\|T\|\leq\|D\|<1$, and
Equation~\eqref{eq:gaussian-defect-action} says that
$I+K_{\omega_T}^2$ acts on the image of $b$ as
$\Lambda=I-T^{\mathsf T}T$.  The second identity in
Equation~\eqref{eq:gaussian-sandwich} therefore gives
\[
 X^{-1/2}YX^{-1/2}=b(\Lambda^{-1/2}\xi).
\]
By cyclicity of trace, Equation~\eqref{eq:gaussian-trace-pairing}
and the isometry of $b$ finish the computation:
\[
 \begin{aligned}
 u&=\tau\!\left(\tfrac12\{b(\xi),X\}
                 b(\Lambda^{-1/2}\xi)\right)\\
  &=\langle b(\xi),b(\Lambda^{-1/2}\xi)\rangle\\
  &=\xi^{\mathsf T}\Lambda^{-1/2}\xi,
 \end{aligned}
\]
which is Equation~\eqref{eq:u-compression}.

For $v$, use $X=X_{\omega_D}=2^NG$, $w=R=a(\gamma)$, and
$S=(I+K_{\omega_D}^2)^{1/2}$.
Equation~\eqref{eq:gaussian-defect-action} identifies
$Sw=a((I-D^2)^{1/2}\gamma)$.  The first identity in
Equation~\eqref{eq:gaussian-sandwich}, followed by
Equation~\eqref{eq:gaussian-trace-pairing}, yields
\[
 v=\tau(X^{1/2}RX^{1/2}R)
  =\gamma^{\mathsf T}(I-D^2)^{1/2}\gamma,
\]
which is Equation~\eqref{eq:v-compression}.
\end{proof}

\begin{remark}
  The computation of $u$ can also be arrived at from considering the overlap of
  displaced fermionic Gaussian states.  The generalized Balian--Br\'ezin formula can compute this overlap in~\cite[Equation~(55)]{SeifiJafarizadehRajabpour2024} after an application of the Colpa trick.
\end{remark}

\subsection{The Chebyshev specialization}

Fix $N\geq2$ and put
\begin{equation}
 \theta_{N,j}=\frac{(j+1/2)\pi}{N},\qquad
 \eta_N=(\cos\theta_{N,j})_{j=0}^{N-1},\qquad
 \gamma_N=\frac1{\sqrt N}(1,\ldots,1)^{\mathsf T}.
 \label{eq:Chebyshev-parameters}
\end{equation}
Thus the entries of $\eta_N$ are the cosines of the midpoints of $N$ equal
subintervals of $[0,\pi]$, with equal weights $\gamma_{N,j}^2=1/N$.  Abbreviate
\[
 G_N=G_{\eta_N},\qquad P_N=P_{\gamma_N}^+,\qquad
 K_N=K_{G_N,P_N},\qquad L_N=L_{G_N,P_N},
\]
and
\[
 D_N=D_{\eta_N},\qquad
 \Lambda_N=\Lambda_{\eta_N,\gamma_N},\qquad
 \xi_N=\xi_{\eta_N,\gamma_N}.
\]

Theorem~\ref{thm:compression} reduces the remaining problem to two scalar
measures.
For the data $(\eta,\gamma)$, write the spectral decomposition
\[
 \Lambda_{\eta,\gamma}=\sum_r\lambda_r\Pi_r
\]
and let
\begin{equation}
 \mu_{\eta,\gamma}=\sum_{j=1}^N\gamma_j^2\delta_{\eta_j},
 \qquad
 \nu_{\eta,\gamma}
 =\sum_r\|\Pi_r\xi_{\eta,\gamma}\|^2\delta_{\lambda_r}.
 \label{eq:quadrature-measure}
\end{equation}
For the data in Equation~\eqref{eq:Chebyshev-parameters}, the first of these
is the discrete probability measure
\begin{equation}
 \mu_N:=\mu_{\eta_N,\gamma_N}
 =\frac1N\sum_{j=0}^{N-1}\delta_{\cos((j+1/2)\pi/N)}.
 \label{eq:Chebyshev-root-measure}
\end{equation}
For every real-valued function $f$ on
$\operatorname{spec}(\Lambda_{\eta,\gamma})$, spectral functional calculus
gives
\begin{equation}
 \int f(\lambda)\,d\nu_{\eta,\gamma}(\lambda)
 =\xi_{\eta,\gamma}^{\mathsf T}
   f(\Lambda_{\eta,\gamma})\xi_{\eta,\gamma}.
 \label{eq:nu-functional-calculus}
\end{equation}
We also have
\[
 v(\eta,\gamma)
 =\sum_{j=1}^N\gamma_j^2\sqrt{1-\eta_j^2}
 =\int_{-1}^1\sqrt{1-t^2}\,d\mu_{\eta,\gamma}(t).
\]
Moreover, $\operatorname{spec}(\Lambda_{\eta,\gamma})\subset(0,1]$, so
spectral functional calculus with $f(\lambda)=\lambda^{-1/2}$ gives
\[
 u(\eta,\gamma)
 =\xi_{\eta,\gamma}^{\mathsf T}
   \Lambda_{\eta,\gamma}^{-1/2}\xi_{\eta,\gamma}
 =\int_{(0,1]}\lambda^{-1/2}\,d\nu_{\eta,\gamma}(\lambda).
\]
Combining these identities yields
\begin{equation}
 \frac{u(\eta,\gamma)+v(\eta,\gamma)}2
 =\frac12\left(
  \int_{-1}^1\sqrt{1-t^2}\,d\mu_{\eta,\gamma}(t)
  +\int_{(0,1]}\lambda^{-1/2}\,d\nu_{\eta,\gamma}(\lambda)
 \right).
 \label{eq:J-two-measures}
\end{equation}
The data in Equation~\eqref{eq:Chebyshev-parameters} are chosen so that both
measures are explicit and both integrals approach $2/\pi$.

The target for the first integral is the arcsine average.  Give
$0<\theta<\pi$ the uniform probability measure
$d\theta/\pi$ and set
$\eta=\cos\theta$.  Its pushforward under
$\theta\mapsto\eta=\cos\theta$ is the arcsine probability measure
\begin{equation}
 d\mu_{\rm arc}(\eta)=\frac{d\eta}{\pi\sqrt{1-\eta^2}}.
 \label{eq:arcsine-measure}
\end{equation}
In particular,
\begin{equation}
 \int_{-1}^1\sqrt{1-\eta^2}\,d\mu_{\rm arc}(\eta)
 =\frac1\pi\int_0^\pi\sin\theta\,d\theta
 =\frac2\pi.
 \label{eq:arcsine-sine-mean}
\end{equation}

The Chebyshev polynomials $T_N$ and $U_{N-1}$ are defined by
\begin{equation}
 T_N(\cos\theta)=\cos(N\theta),\qquad
 U_{N-1}(\cos\theta)=\frac{\sin(N\theta)}{\sin\theta};
 \label{eq:Chebyshev-definitions}
\end{equation}
in particular, $T_N'=NU_{N-1}$.

\begin{lemma}
\label{lem:Chebyshev}
For the data in Equation~\eqref{eq:Chebyshev-parameters},
\begin{equation}
 \nu_N:=\nu_{\eta_N,\gamma_N}
 =\frac1N\sum_{k=1}^{N-1}\sin^2\!\frac{k\pi}{N}\,
 \delta_{\sin^2(k\pi/N)}.
 \label{eq:Chebyshev-measure}
\end{equation}
\end{lemma}

\begin{proof}
We identify the atoms and weights of $\nu_N$ by computing
$\xi_N^{\mathsf T}(\Lambda_N-zI)^{-1}\xi_N$.  Writing
$\Lambda_N=\sum_r\lambda_r\Pi_r$, the spectral decomposition gives
\begin{equation}
 \xi_N^{\mathsf T}(\Lambda_N-zI)^{-1}\xi_N
 =\sum_r\frac{\|\Pi_r\xi_N\|^2}{\lambda_r-z}.
 \label{eq:nu-spectral-resolvent}
\end{equation}
We use Sherman--Morrison and logarithmic derivatives of the Chebyshev
polynomials to show that the same quantity equals
\[
 \frac1N\sum_{k=1}^{N-1}
 \frac{\sin^2(k\pi/N)}{\sin^2(k\pi/N)-z}.
\]
Comparing the two expressions proves the claimed formula for $\nu_N$.

Let $S_N^2=I-D_N^2$.  Since
$\Lambda_N=S_N^2+\xi_N\xi_N^{\mathsf T}$, Sherman--Morrison gives
\begin{equation}
 \xi_N^{\mathsf T}(\Lambda_N-zI)^{-1}\xi_N
 =\frac{h_N(z)}{1+h_N(z)},
 \qquad h_N(z)=\xi_N^{\mathsf T}(S_N^2-zI)^{-1}\xi_N.
 \label{eq:SM}
\end{equation}
Because $S_N^2$ is diagonal and
$\xi_N=N^{-1/2}(\eta_{N,0},\ldots,\eta_{N,N-1})^{\mathsf T}$,
\begin{equation}
 h_N(z)=\frac1N\sum_{j=0}^{N-1}
 \frac{\eta_{N,j}^2}{1-\eta_{N,j}^2-z}.
 \label{eq:hN-diagonal}
\end{equation}
Put $x^2=1-z$.  The roots of $T_N$ are the $\eta_{N,j}$, and
they are invariant as a set under $\eta\mapsto-\eta$.  Hence the
logarithmic derivative of $T_N$ gives
\begin{align}
 \sum_{j=0}^{N-1}\frac1{x^2-\eta_{N,j}^2}
 &=\frac1{2x}\sum_{j=0}^{N-1}
   \left(\frac1{x-\eta_{N,j}}+\frac1{x+\eta_{N,j}}\right)\notag\\
 &=\frac{T_N'(x)}{xT_N(x)}.
 \label{eq:Chebyshev-squared-log-derivative}
\end{align}
Since
$\eta^2/(x^2-\eta^2)=x^2/(x^2-\eta^2)-1$,
Equations~\eqref{eq:hN-diagonal}
and~\eqref{eq:Chebyshev-squared-log-derivative} imply
\[
 h_N(z)
 =\frac{xT_N'(x)}{NT_N(x)}-1
 =\frac{xU_{N-1}(x)}{T_N(x)}-1,
\]
where the last equality uses $T_N'=NU_{N-1}$.  Therefore
\begin{equation}
 \frac{h_N(z)}{1+h_N(z)}
 =1-\frac{T_N(x)}{xU_{N-1}(x)}.
\end{equation}
The roots of $U_{N-1}$, $x_k=\cos\frac{k\pi}{N}$, are invariant as a set
under $x_k\mapsto-x_k$.  The logarithmic derivative of $U_{N-1}$ gives
\[
 \sum_{k=1}^{N-1}\frac1{x^2-x_k^2}
 =\frac{U_{N-1}'(x)}{xU_{N-1}(x)}.
\]
Since $1-x_k^2=(1-x^2)+(x^2-x_k^2)$, it follows that
\begin{equation}
 \frac1N\sum_{k=1}^{N-1}\frac{1-x_k^2}{x^2-x_k^2}
 =\frac{N-1}{N}
   +\frac{1-x^2}{N}\frac{U_{N-1}'(x)}{xU_{N-1}(x)} =1-\frac{T_N(x)}{xU_{N-1}(x)}.
 \label{eq:Chebyshev-critical-log-derivative}
\end{equation}
Here the last equality is the Chebyshev differential equation in the form
\[
 (1-x^2)U_{N-1}'(x)=xU_{N-1}(x)-NT_N(x).
\]
Finally, $1-x_k^2=\sin^2(k\pi/N)$ and
$x^2-x_k^2=\sin^2(k\pi/N)-z$, so
\begin{equation}
 1-\frac{T_N(x)}{xU_{N-1}(x)}
 =\frac1N\sum_{k=1}^{N-1}
 \frac{\sin^2(k\pi/N)}{\sin^2(k\pi/N)-z}.
\end{equation}
Together with Equation~\eqref{eq:nu-spectral-resolvent}, this proves
Equation~\eqref{eq:Chebyshev-measure}.
\end{proof}

Equations~\eqref{eq:Chebyshev-root-measure}
and~\eqref{eq:Chebyshev-measure} now allow us to evaluate $u_N$ and $v_N$
explicitly.

\begin{theorem}
\label{thm:explicit-recovery}
For this specialization, write
\[
 u_N=u(\eta_N,\gamma_N),\qquad v_N=v(\eta_N,\gamma_N).
\]
Then
\begin{align}
 u_N&=\frac1N\cot\frac{\pi}{2N},
 \label{eq:uN}\\
 v_N&=\frac1N\csc\frac{\pi}{2N},
 \label{eq:vN}\\
 J(K_N,L_N)&=\frac1{2N}\cot\frac{\pi}{4N}.
 \label{eq:JN}
\end{align}
Consequently
\begin{equation}
 J(K_N,L_N)\uparrow\frac2\pi,
 \qquad
 \frac2\pi-J(K_N,L_N)
 =\frac{\pi}{24N^2}+O(N^{-4}).
 \label{eq:J-rate}
\end{equation}
\end{theorem}

\begin{proof}
Lemma~\ref{lem:Chebyshev} evaluates the second integral in
Equation~\eqref{eq:J-two-measures}:
\[
 u_N=\frac1N\sum_{k=1}^{N-1}\sin\frac{k\pi}{N}
 =\frac1N\cot\frac{\pi}{2N}.
\]
Equation~\eqref{eq:Chebyshev-root-measure} gives the first integral:
\[
 v_N=\frac1N\sum_{j=0}^{N-1}\sin\theta_{N,j}
 =\frac1N\csc\frac{\pi}{2N}.
\]
The sine sums are imaginary parts of geometric series.  Since
$J(K_N,L_N)=(u_N+v_N)/2$ and
$\cot x+\csc x=\cot(x/2)$,
Equation~\eqref{eq:JN} follows.  The result now follows from the Laurent
expansion of the cotangent and the monotonicity of $x\cot x$ as
$x\to0^+$.
\end{proof}

\section{The upper bound}
\label{sec:closing}

Section~\ref{sec:recovery} gives normalized Horn pairs $(K_N,L_N)$ with
\[
 J(K_N,L_N)\uparrow\frac2\pi.
\]
Gram symmetrization turns each of them into a pair with $\Phi=1/2$ and
$\Gamma=J(K_N,L_N)$.

\begin{lemma}
\label{lem:right-of-endpoint-upper}
For $\delta_\pi<\delta<1/2$,
\begin{equation}
 R_D(\delta)\leq\frac\pi4(1-2\delta).
 \label{eq:right-of-endpoint-upper}
\end{equation}
\end{lemma}

\begin{proof}
Write $J_N=J(K_N,L_N)$.  By Proposition~\ref{prop:duality}, the
Gram-symmetrized pair associated with $(K_N,L_N)$ satisfies
\[
 \Gamma(\widehat K_N,\widehat L_N)=J_N,
 \qquad
 \Phi(\widehat K_N,\widehat L_N)=\frac12.
\]
For all sufficiently large $N$, one has $J_N>1-2\delta$, so
Lemma~\ref{lem:Horn-dilution} gives
\[
 R_D(\delta)\leq\frac{1-2\delta}{2J_N}.
\]
Letting $N\to\infty$ proves Equation~\eqref{eq:right-of-endpoint-upper}.
\end{proof}

\begin{theorem}[Upper bound]
\label{thm:upper}
We have $R_D(\delta_\pi)\leq1/2$, and
$R_D(\delta)<\frac12$ for every $\delta>\delta_\pi$.
\end{theorem}

\begin{proof}
Lemma~\ref{lem:right-of-endpoint-upper} gives
$R_D(\delta)<1/2$ for $\delta_\pi<\delta<1/2$, and monotonicity extends
this strict bound to every $\delta>\delta_\pi$.

For $\delta_\pi<\delta<1/2$, Lemma~\ref{lem:shortening} gives
\[
 R_D(\delta_\pi)
 \leq1-\frac{\delta_\pi}{\delta}
 \bigl(1-R_D(\delta)\bigr)
 \leq1-\frac{\delta_\pi}{\delta}
 \left(1-\frac\pi4(1-2\delta)\right).
\]
Letting $\delta\downarrow\delta_\pi$ proves the endpoint bound.
\end{proof}

Corollary~\ref{cor:lower} and Theorem~\ref{thm:upper} prove
Theorem~\ref{thm:main}.

\begin{proof}[Proof of Theorem~\ref{thm:Krawtchouk-sign-uncertainty}]
Proposition~\ref{prop:Krawtchouk-sign-lower} gives the lower limit for both
signs.  For the reverse inequality, fix
$\delta\in(\delta_\pi,1/2)$ and put $d_n=\lceil\delta n\rceil$.
Theorem~\ref{thm:upper} gives $R_D(\delta)<1/2$, so for all sufficiently
large $n$,
\begin{equation}
 \operatorname{LP}_n(d_n)<2^{n/2}.
 \label{eq:LP-below-balanced-scale}
\end{equation}
Alternative \textup{(ii)} in Lemma~\ref{lem:finite-sign-alternative} would
give a Delsarte-feasible point of objective $2^{n/2}$, contradicting
Equation~\eqref{eq:LP-below-balanced-scale}.  Alternative \textup{(i)}
therefore holds, and
\[
 A^{\mathrm K}_-(n)\leq d_n.
\]
It follows that
$\limsup_n A^{\mathrm K}_-(n)/n\leq\delta$.  Letting
$\delta\downarrow\delta_\pi$ proves the matching upper limit for the
negative sign.  Proposition~\ref{prop:Krawtchouk-sign-comparison} transfers it to
the positive sign.  Together with the common lower bound, this proves
Equation~\eqref{eq:Krawtchouk-sign-radius}.
\end{proof}

\section*{Acknowledgements}

The author thanks Henry Cohn for early discussions, several years ago, about
the appropriate formulation of a binary coding analogue of the sphere-packing
conjecture \cite[Conjecture~3.2]{AfkhamiJeddiEtAl2020}.

\bibliographystyle{alpha}
\bibliography{references}

\end{document}